\documentclass[a4paper,twoside,10pt]{article}

\usepackage[english]{babel}
\usepackage[latin1]{inputenc}
\usepackage{lmodern,color}
\usepackage[T1]{fontenc}
\usepackage{indentfirst}
\usepackage{amsmath,amssymb}

\usepackage[a4paper,top=3cm,bottom=3cm,left=3cm,right=3cm,%
bindingoffset=5mm]{geometry}
\usepackage{lmodern}
\usepackage[T1]{fontenc}
\usepackage{lettrine}
\usepackage{booktabs}

\usepackage{authblk}
\usepackage{emp}
\usepackage{amsthm}
\usepackage{amsmath,amssymb,amsthm}
\usepackage{braket}
\usepackage{chngpage}
\usepackage{cases}
\usepackage{graphicx}
\usepackage{epstopdf}
\usepackage{tabularx}
\usepackage{multirow}

\newcommand{\numberset}{\mathbb}
\newcommand{\N}{\numberset{N}}
\newcommand{\R}{\numberset{R}}

\newcommand{\B}{\numberset{B}}
\newcommand{\Pk}{\numberset{P}}
\renewcommand{\epsilon}{\varepsilon}
\renewcommand{\theta}{\vartheta}
\renewcommand{\rho}{\varrho}
\renewcommand{\phi}{\varphi}
\newcommand{\uu}{\mathbf{u}}
\newcommand{\vv}{\mathbf{v}}
\newcommand{\ww}{\mathbf{w}}
\newcommand{\ff}{\mathbf{f}}
\newcommand{\te}{\boldsymbol{\theta}}
\newcommand{\xx}{\boldsymbol{\chi}}
\newcommand{\fg}{\boldsymbol{\phi}}
\newcommand{\fgl}{\boldsymbol{\psi}}
\newcommand{\ddd}{\boldsymbol{\delta}}
\newcommand{\dd}{{\rm div}}
\newcommand{\gr}{\nabla}
\newcommand{\Gr}{\boldsymbol{\nabla}}
\newcommand{\rr}{{\rm rot}}
\newcommand{\cc}{\boldsymbol{{\rm curl}}}
\newcommand{\dl}{\boldsymbol{\Delta}}
\newcommand{\VV}{\mathbf{V}}
\newcommand{\WW}{\mathbf{W}}
\newcommand{\ZZ}{\mathbf{Z}}
\newcommand{\cskew}{\widetilde{c}}
\def\P0{{\Pi^{0, E}_k}}

\def\PN{{\Pi^{\nabla, E}_k}}
\def\PP0{{\boldsymbol{\Pi}^{0, E}_{k-1}}}
\def\ggp{\mathbf{g}_{k-2}^{\oplus}}
\def\ggq{\mathbf{g}_{k}^{\oplus}}
\def\GGp{\mathcal{G}_{k-2}^{\oplus}(E)}
\def\GGq{\mathcal{G}_{k}^{\oplus}(E)}

{\left\lbrace\begin{array}{@{}l@{}}}%
{\end{array}\right.}
\theoremstyle{definition}

\theoremstyle{remark}
\newtheorem{remark}{Remark}[section]

\theoremstyle{remark}
\newtheorem{test}{Test}[section]

\theoremstyle{plain}
\newtheorem{theorem}{Theorem}[section]
\newtheorem{proposition}{Proposition}[section]

\newtheorem{lemma}{Lemma}[section]

\usepackage{lipsum}

\usepackage{authblk}

\usepackage{setspace}

\author[1]{L. Beir\~ao da Veiga \thanks{lourenco.beirao@unimib.it}}

\author[2]{C. Lovadina \thanks{carlo.lovadina@unimi.it}}

\author[1]{G. Vacca \thanks{giuseppe.vacca@unimib.it}}

\affil[1]{Dipartimento di Matematica e Applicazioni,  Universit\`a degli Studi di Milano Bicocca, Via Roberto Cozzi 55 - 20125 Milano, Italy}  

\affil[2]{Dipartimento di Matematica,  Universit\`a degli Studi di Milano, Via Cesare Saldini 50 - 20133 Milano, Italy}

\title{\textbf{Virtual Elements for the Navier-Stokes problem on polygonal meshes}}

\date{\today}

\begin{document}

\maketitle
\begin{abstract}
A family of Virtual Element Methods for the 2D Navier-Stokes equations is proposed and analysed. The schemes provide a discrete velocity field which is point-wise divergence-free. A rigorous error analysis is developed, showing that the methods are stable and optimally convergent. Several numerical tests are presented, confirming the theoretical predictions. A comparison with some mixed finite elements is also performed.
\end{abstract}

\section{Introduction}
\label{sec:1}

The Virtual Element Method (VEM), introduced in \cite{volley,VEM-hitchhikers}, is a recent paradigm for the approximation of partial differential equation problems that shares the same variational background of the Finite Element Methods. The original motivation of VEM is the need to construct an accurate {\em conforming} Galerkin scheme with the capability to deal with highly general polygonal/polyhedral meshes, including ``hanging vertexes'' and non-convex shapes. Among the Galerkin schemes, VEM has the peculiarity that the discrete spaces consist of functions which are not known pointwise, but a limited set of information about them are at disposal. Nevertheless, the available information are sufficient   
to construct the stiffness matrix and the right-hand side. 

The VEM has been developed for many problems, see for example \cite{Brezzi-Falk-Marini,Ahmed-et-al:2013,variable-primale,Steklov-VEM,Helmholtz-VEM,supermisti,Berrone-VEM,Benedetto-VEM-2,Berrone-SUPG,nonconforming,Manzini:Russo:Sukumar,plates-zhao,vaccahyper,gardini,senatore}. 
Regarding more directly the Stokes problem, Virtual Elements have been developed in \cite{Antonietti-BeiraodaVeiga-Mora-Verani:20XX,Stokes:nonconforme,Stokes:divfree,Gatica-1}.
Moreover, VEM is experiencing a growing interest also towards Continuum Mechanics problems in the engineering community. We here cite the recent works \cite{GTP14,BeiraoLovaMora,ABLS_part_I,wriggers,BCP,Andersen-geo,Topology-VEM} and \cite{BeiraodaVeiga-Brezzi-Marini:2013,Brezzi-Marini:2012,ADLP-HR}, for instance.       
Finally, some example of other numerical methods for the Stokes or Navier-Stokes equations that can handle polytopal meshes are \cite{dipietro_lemaire,qiu_shi,dipietro_krell}.

In this paper, which may be considered as a natural evolution of our recent divergence-free approach developed in \cite{Stokes:divfree} for the Stokes problem, we apply the VEM to the Navier-Stokes equations in 2D. However, the non-linear convective term in the Navier-Stokes equations leads to the introduction of suitable projectors. These, in turn, suggest to make use of an enhanced discrete velocity space \cite{preprintdarcy}, that is an improvement with respect to that of \cite{Stokes:divfree}.
Instead, the pressure field is approximated by means of standard locally polynomial  functions, without any continuity requirement across the elements. Furthermore, we consider two different discretization of the trilinear form arising from the convective term. The first one is the straightforward VEM version of the continuous trilinear form; however, the projector introduction causes a lack of skew-symmetry, despite the discrete velocity is divergence-free (up to machine precision). This leads to consider the second choice, which is simply the skew-symmetric part of the trilinear form mentioned above (cf. \cite{giraultbook}, for instance). We remark that we develop an error analysis focusing on this latter choice, but the numerical tests concern both alternatives. 
The outcome is a family of Virtual Elements, one per each polynomial order of consistency $k$, with $k\ge 2$.  To our best knowledge, this is the first paper where the VEM technology is addressed to the Navier-Stokes equations.       

The main objectives of the present paper are the following.

\begin{itemize}

\item {\em The development of a rigorous error analysis of the proposed methods.} 
We highlight that our analysis provides some noteworthy element of novelty. Indeed, although we follow rather well-established lines for the error analysis of 2D Navier-Stokes Galerkin methods (see for example \cite{giraultbook}), these need to be combined with new techniques that are peculiar to the VEM framework. In particular, the interpolant construction of Theorem \ref{thm:interpolante} involves new arguments which might be useful even in different contexts (i.e. for other VEM spaces with different regularity requirement).

\item {\em A first but thorough assessment of the actual numerical performance of this new approach.} We provide a set of significant numerical tests, that highlight the features of our VEM approach. In addition to the important flexibility of dealing with general polygonal meshes, the presented scheme (we tested the case $k=2$) displays the following favourable points.

\begin{enumerate}

\item The error components partly decouple: notably, the velocity error does not depend directly on the discrete pressures, but only indirectly through the approximation of the loading and convection terms. This is a consequence of the fact that our methods provide a discrete velocity which is point-wise divergence-free (the isochoric constraint is {\em not} relaxed). 
In some situations, e.g. for hydrostatic fluid problems, the partial decoupling of the errors induces a positive effect on the velocity approximation. 
Moreover, for the same reason, the VEM scheme seems to be more robust for small values of the viscosity parameter when compared with standard mixed finite elements. 

\item Another advantage of the method is that, again due to its divergence-free nature, the same Virtual space couple can be used directly also for the approximation of the diffusion problem (in mixed form). This allows for a much easier coupling in Stokes-Darcy problems where different models need to be used in different parts of the domain. This observation adds up with the fact that, thanks to the use of polygons that allow hanging nodes, the gluing of different meshes in different parts of the domain is also much easier.


\item As in \cite{Stokes:divfree}, the particular choice of degrees of freedom adopted for the velocity space yields a diagonal structure in a large part of the pressure-velocity interaction stiffness matrix. As a consequence, and without the need of any static condensation, many internal-to-element degrees of freedom can be automatically ignored when building the linear system.

\end{enumerate}
 	
\end{itemize}

We finally note that, nowadays, there do exist Galerkin-type finite element methods for the Stokes and Navier-Stokes equations that are pressure-robust (that is, the error on the velocity does not depend on the pressure, not even indirectly through the loading or convection terms). Some recent examples are \cite{benchmark,dipietro_linke}.
However, to our best knowledge, all the available schemes work only for standard simplicial/hexahedral meshes. 
Despite our method is not pressure-robust in the sense above, for arbitrary polygonal meshes it is the only conforming divergence-free scheme, a property which yields to important advantages, as outlined in points 1 and 2. Developing a conforming scheme which is both divergence-free and pressure-robust for general polygonal meshes, is currently an open problem.


A brief outline of the paper is the following.  In Section \ref{sec:2} we recall the 2D Navier-Stokes problem, introducing the classical variational formulation and the necessary notations. Section \ref{sec:3} details the proposed discretization procedure. The approximation spaces and all the quantities that form the discrete problem, are introduced and described. 
Section \ref{sec:4} deals with the theoretical analysis, which leads to the optimal error estimates of Theorem \ref{thm:u} and bound \eqref{eq:p-est}.
Finally, Section \ref{sec:5} presents several numerical tests, which highlight the actual performance of our approach, also in comparison with a couple of well-known mixed finite element schemes.

\section{The continuous Navier-Stokes equation}
\label{sec:2}

We consider the steady Navier-Stokes Equation on a polygonal domain $\Omega \subseteq \R^2$ with homogeneous Dirichlet boundary conditions:
\begin{equation}
\label{eq:ns primale}
\left\{
\begin{aligned}
& \mbox{ find $(\uu,p)$ such that}& &\\
& -  \nu \, \dl \uu + (\Gr \uu ) \,\uu -  \nabla p = \ff\qquad  & &\text{in $\Omega$,} \\
& \dd \, \uu = 0 \qquad & &\text{in $\Omega$,} \\
& \uu = 0  \qquad & &\text{on $\Gamma = \partial \Omega$,}
\end{aligned}
\right.
\end{equation}
with $\nu \in {\mathbb R}$, $\nu > 0$, and where $\uu, p$ are the velocity and the pressure fields, respectively. 
Furthermore, $\dl $, $\dd$, $\Gr$, and $\gr$ denote the vector Laplacian,  
the divergence, the gradient operator for vector fields and the gradient operator for scalar functions. Finally, $\ff$ represents 
the external force, while $\nu$ is the viscosity. We also remark that different boundary conditions can be treated as well. 

Let us consider the spaces
\begin{equation}
\label{eq:spazi continui}
\VV:= \left[ H_0^1(\Omega) \right]^2, \qquad Q:= L^2_0(\Omega) = \left\{ q \in L^2(\Omega) \quad \text{s.t.} \quad \int_{\Omega} q \,{\rm d}\Omega = 0 \right\} 
\end{equation}
with norms
\begin{equation}
\label{eq:norme continue}
\| \vv \|_{\VV} := | \vv|_{\left[ H^1(\Omega) \right]^2} \quad , \qquad 
\|q\|_Q := \| q\|_{L^2(\Omega)}. 
\end{equation}

We assume $\ff \in[L^{2}(\Omega)]^2$ and consider the bilinear forms 
\begin{gather}
\label{eq:forma a}
a(\cdot, \cdot) \colon \VV \times \VV \to \R,    \qquad 
a (\mathbf{u},  \mathbf{v}) := \int_{\Omega}  \, \boldsymbol{\nabla}   \mathbf{u} : \boldsymbol{\nabla}  \mathbf{v} \,{\rm d} \Omega, \qquad \text{for all $\mathbf{u},  \mathbf{v} \in \mathbf{V}$}
\\
\label{eq:forma b}
b(\cdot, \cdot) \colon \mathbf{V} \times Q \to \R \qquad b(\mathbf{v}, q) :=  \int_{\Omega}q\, {\rm div} \,\mathbf{v} \,{\rm d}\Omega \qquad \text{for all $\mathbf{v} \in \mathbf{V}$, $q \in Q$}
\\
\label{eq:forma c}
c(\cdot; \, \cdot, \cdot) \colon \VV \times \VV \times \VV \to \R \qquad c(\ww; \, \uu, \vv) :=  \int_{\Omega} ( \Gr \uu ) \, \ww \cdot \vv  \,{\rm d}\Omega \qquad \text{for all $\ww, \uu, \vv \in \VV$.}
\end{gather}
Then a standard variational formulation of Problem \eqref{eq:ns primale} is:
\begin{equation}
\label{eq:ns variazionale}
\left\{
\begin{aligned}
& \text{find $(\uu, p) \in \VV \times Q$, such that} \\
& \nu \, a(\uu, \vv) + c(\uu; \, \uu, \vv) + b(\vv, p) = (\ff, \vv) \qquad & \text{for all $\vv \in \VV$,} \\
&  b(\uu, q) = 0 \qquad & \text{for all $q \in Q$,}
\end{aligned}
\right.
\end{equation}
where
\[
(\ff, \vv) := \int_{\Omega} \ff \cdot \vv \, {\rm d} \Omega .
\]
It is well known that with the choices \eqref{eq:norme continue}, we have (see for instance \cite{giraultbook}):
\begin{itemize}
\item $a(\cdot, \cdot)$, $b(\cdot, \cdot)$ and $c(\cdot; \, \cdot, \cdot)$ are continuous
\[
|a(\uu, \vv)| \leq  \|\uu\|_{\VV}\|\vv\|_{\VV} \qquad \text{for all $\uu, \vv \in \VV$,}
\]
\[
|b(\vv, q)| \leq  \|\vv\|_{\VV} \|q\|_Q \qquad \text{for all $\vv \in \VV$ and $q \in Q$,}
\]
\[
|c(\ww; \, \uu, \vv)| \leq \widehat{C} \, \|\ww\|_{\VV} \|\uu\|_{\VV} \|\vv\|_{\VV} 
\qquad \text{for all $\ww, \uu, \vv \in \VV$;}
\]
\item $a(\cdot, \cdot)$ is coercive (with coercivity constant $\alpha =1$), i.e. 
\[
a(\mathbf{v}, \mathbf{v}) \geq  \|\mathbf{v}\|^2_{\VV} \qquad \text{for all $\mathbf{v} \in \mathbf{V}$;}
\]
\item  the bilinear form $b(\cdot,\cdot) $ and the space $\VV$ and $Q$ satisfy the inf-sup condition, i.e.   
\begin{equation}
\label{eq:inf-sup}
\exists \, \beta >0 \quad \text{such that} \quad \sup_{\mathbf{v} \in \mathbf{V}, \, \mathbf{v} \neq \mathbf{0}} \frac{b(\mathbf{v}, q)}{ \|\mathbf{v}\|_{\VV}} \geq \beta \|q\|_Q \qquad \text{for all $q \in Q$.}
\end{equation}
\end{itemize}
Therefore, if
\begin{equation}
\label{eq:ns condition}
\gamma := \frac{\widehat{C} \, \|\ff\|_{-1}}{\nu^2} < 1
\end{equation}
 Problem \eqref{eq:ns variazionale} has a unique solution $(\mathbf{u}, p) \in \mathbf{V} \times Q$ such that
\begin{equation}
\label{eq:solution estimates}
\| \uu\|_{\VV} \leq \frac{\| \ff\|_{H^{-1}}}{\nu}.
\end{equation}
%
Let us introduce the kernel
\begin{equation}
\label{eq:kercontinuo}
\ZZ := \{ \vv \in \VV \quad \text{s.t.} \quad b(\vv, q) = 0 \quad \text{for all $q \in Q$}. \}
\end{equation}
Then Problem \eqref{eq:ns variazionale} can be formulated in the equivalent kernel form
\begin{equation}
\label{eq:ns variazionale ker}
\left\{
\begin{aligned}
& \text{find $\uu \in \ZZ$, such that} \\
& \nu \, a(\uu, \vv) + c(\uu; \, \uu, \vv) = (\ff, \vv) \qquad & \text{for all $\vv \in \ZZ$,} 
\end{aligned}
\right.
\end{equation}
Finally, by a direct computation it is easy to see that, if $\uu \in \ZZ$ is fixed, then the bilinear form $c(\uu; \,\cdot, \,\cdot) \colon \VV \times \VV \to \R$ is skew-symmetric, i.e. 
\[
c(\uu; \vv, \ww) = -c(\uu; \ww, \vv) \qquad \text{for all $\vv, \ww \in \VV$} .
\]
Therefore we introduce, as usual, also the trilinear form $\cskew(\cdot, \, \cdot, \, \cdot)\colon \VV \times \VV \times \VV \to \R$ 
\begin{equation}
\label{eq:skew}
\cskew(\ww; \, \uu, \vv) :=  \frac{1}{2}  c(\ww; \, \uu, \vv) - \frac{1}{2} c(\ww; \, \vv, \uu) \qquad \text{for all $\ww, \uu, \vv \in \VV$.}
\end{equation}


\section{Virtual formulation of the problem}
\label{sec:3}

\subsection{Virtual element space and polynomial projections}
\label{sub:3.1}
 
We outline the Virtual Element discretization of Problem \eqref{eq:ns variazionale}. 
We will make use of various tools from the virtual element technology, that will be described briefly; we refer the interested reader to the papers \cite{Stokes:divfree,preprintdarcy}.

Let $\set{\Omega_h}_h$ be a sequence of decompositions of $\Omega$ into general polygonal elements $E$ with
\[
 h_E := {\rm diameter}(E) , \quad
h := \sup_{E \in \Omega_h} h_E .
\]
We suppose that for all $h$, each element $E$ in $\Omega_h$ fulfils the following assumptions:
\begin{description}
\item [$\mathbf{(A1)}$] $E$ is star-shaped with respect to a ball $B_E$ of radius $ \geq\, \rho \, h_E$, 
\item [$\mathbf{(A2)}$] the distance between any two vertexes of $E$ is $\ge c \, h_E$, 
\end{description}
where $\rho$ and $c$ are positive constants. We remark that the hypotheses above, though not too restrictive in many practical cases, 
can be further relaxed, as investigated in ~\cite{2016stability}. 
Using standard VEM notation, for $k \in \N$, let us define the spaces 
\begin{itemize}
\item $\Pk_k(E)$ the set of polynomials on $E$ of degree $\leq k$  (with the extended notation $\Pk_{-1}(E)=\emptyset$),
\item $\B_k(E) := \{v \in C^0(\partial E) \quad \text{s.t} \quad v_{|e} \in \Pk_k(e) \quad \forall\mbox{ edge } e \subset \partial E\}$,
\item $\mathcal{G}_{k}(E):= \nabla(\Pk_{k+1}(E)) \subseteq [\Pk_{k}(E)]^2$,
\item $\mathcal{G}_{k}^{\oplus}(E) := \mathbf{x}^{\perp}[\Pk_{k-1}(E)] \subseteq [\Pk_{k}(E)]^2$ with $\mathbf{x}^{\perp}:= (x_2, -x_1)$.
\end{itemize}
For any $n \in \N$ and $E \in \Omega_h$ we introduce the following useful polynomial projections:
\begin{itemize}

\item the $\boldsymbol{H^1}$ \textbf{semi-norm projection} ${\Pi}_{n}^{\nabla,E} \colon \VV \to [\Pk_n(E)]^2$, defined by 
\begin{equation}
\label{eq:Pn_k^E}
\left\{
\begin{aligned}
& \int_E \boldsymbol{\nabla} \,\mathbf{q}_n : \boldsymbol{\nabla} (\mathbf{v}- \, {\Pi}_{n}^{\nabla,E}   \mathbf{v}) \, {\rm d} E = 0 \qquad  \text{for all $\vv \in \VV$ and for all $\mathbf{q}_n \in [\Pk_n(E)]^2$,} \\
& \Pi_0^{0,E}(\mathbf{v} - \,  {\Pi}_{n}^{\nabla, E}  \mathbf{v}) = \mathbf{0} \, ,
\end{aligned}
\right.
\end{equation} 

\item the $\boldsymbol{L^2}$\textbf{-projection for scalar functions} $\Pi_n^{0, E} \colon L^2(\Omega) \to \Pk_n(E)$, given by
\begin{equation}
\label{eq:P0_k^E}
\int_Eq_n (v - \, {\Pi}_{n}^{0, E}  v) \, {\rm d} E = 0 \qquad  \text{for all $v \in L^2(E)$  and for all $q_n \in \Pk_n(E)$,} 
\end{equation} 
with obvious extension for vector functions $\Pi_n^{0, E} \colon [L^2(\Omega)]^2 \to [\Pk_n(E)]^2$, and tensor functions 
$\boldsymbol{\Pi}_{n}^{0, E} \colon [L^2(E)]^{2 \times 2} \to [\Pk_{n}(E)]^{2 \times 2}$. 


\end{itemize}
In \cite{Stokes:divfree} we have introduced a new family of Virtual Elements for the Stokes problem on polygonal meshes. In particular, by a proper choice of the Virtual space of velocities, the virtual local spaces are associated to a Stokes-like variational problem on each element. In \cite{preprintdarcy} we have presented an enhanced Virtual space, taking the inspiration from \cite{Ahmed-et-al:2013}, to be used in place of the original one in such a way that the $L^2$-projection can be exactly computable by the DoFs.
In this section we briefly  recall from \cite{Stokes:divfree, preprintdarcy} the notations, the main properties of the Virtual spaces and some details about the construction of the projections.

Let $k \geq 2$ the polynomial degree of accuracy of the method. We introduce  on each element $E \in \Omega_h$ the (original) finite dimensional local virtual space \cite{Stokes:divfree}
\begin{multline}
\label{eq:W_h}
\mathbf{W}_h^E := \biggl\{  
\mathbf{v} \in [H^1(E)]^2 \quad \text{s.t} \quad \mathbf{v}_{|{\partial E}} \in [\B_k(\partial E)]^2 \, , \biggr.
\\
\left.
\biggl\{
\begin{aligned}
& - \boldsymbol{\Delta}    \mathbf{v}  -  \nabla s \in \mathcal{G}_{k-2}^{\oplus}(E),  \\
& {\rm div} \, \mathbf{v} \in \Pk_{k-1}(E),
\end{aligned}
\biggr. \qquad \text{ for some $s \in L^2(E)$}
\quad \right\}
\end{multline}
where all the operators and equations above are to be interpreted in the distributional sense.
Then we enlarge the previous space
\begin{multline*}
\mathbf{U}_h^E := \biggl\{  
\mathbf{v} \in [H^1(E)]^2 \quad \text{s.t} \quad \mathbf{v}_{|{\partial E}} \in [\B_k(\partial E)]^2 \, , \biggr.
\\
\left.
\biggl\{
\begin{aligned}
& - \boldsymbol{\Delta}    \mathbf{v}  -  \nabla s \in \mathcal{G}_{k}^{\oplus}(E),  \\
& {\rm div} \, \mathbf{v} \in \Pk_{k-1}(E),
\end{aligned}
\biggr. \qquad \text{ for some $s \in L^2(E)$}
\quad \right\}
\end{multline*}
Now we define the  Virtual Element space $\mathbf{V}_h^E$ as the restriction  of $\mathbf{U}_h^E$ given by
\begin{equation}
\label{eq:V_h^E}
\mathbf{V}_h^E := \left\{ \mathbf{v} \in \mathbf{U}_h^E \quad \text{s.t.} \quad   \left(\mathbf{v} - \Pi^{\nabla,E}_k \mathbf{v}, \, \mathbf{g}_k^{\perp} \right)_{[L^2(E)]^2} = 0 \quad \text{for all $\mathbf{g}_k^{\perp} \in  \mathcal{G}_{k}^{\oplus}(E)/\mathcal{G}_{k-2}^{\oplus}(E)$} \right\} ,
\end{equation}
where the symbol $\mathcal{G}_{k}^{\oplus}(E)/\mathcal{G}_{k-2}^{\oplus}(E)$ denotes the polynomials in $\mathcal{G}_{k}^{\oplus}(E)$ that are $L^2$-orthogonal to all polynomials of $\mathcal{G}_{k-2}^{\oplus}(E)$ (observing that $\mathcal{G}_{k-2}^{\oplus}(E) \subset \mathcal{G}_{k}^{\oplus}(E)$).
From \cite{supermisti,Stokes:divfree,preprintdarcy}, we recall the following properties of the space $\VV_h^E$. 
The proof of the following result can be found in \cite{preprintdarcy}.
\begin{proposition}[Dimension and DoFs]
\label{prp:dofs}
Let $\VV_h^E$ be the space defined in \eqref{eq:V_h^E}. Then the dimension and $\VV_h^E$ is
\begin{equation}
\label{eq:dimensione V_h^E}
\begin{split}
\dim\left( \mathbf{V}_h^E \right) &= \dim\left([\B_k(\partial E)]^2\right) + \dim\left(\mathcal{G}_{k-2}^{\oplus}(E)\right) + \left( \dim(\Pk_{k-1}(E)) - 1\right) \\
&= 2n_E k + \frac{(k-1)(k-2)}{2}  + \frac{(k+1)k}{2} - 1.
\end{split}
\end{equation}
where $n_E$ is the number of vertexes of $E$. Moreover the following linear operators $\mathbf{D_V}$, split into four subsets (see Figure \ref{fig:dofsloc}) constitute a set of DoFs for $\VV_h^E$:
\begin{itemize}
\item $\mathbf{D_V1}$:  the values of $\mathbf{v}$ at the vertices of the polygon $E$,
\item $\mathbf{D_V2}$: the values of $\mathbf{v}$ at $k-1$ distinct points of every edge $e \in \partial E$,
\item $\mathbf{D_V3}$: the moments of $\mathbf{v}$
\[
\int_E \mathbf{v} \cdot \mathbf{g}_{k-2}^{\oplus} \, {\rm d}E \qquad \text{for all $\mathbf{g}_{k-2}^{\oplus} \in \mathcal{G}_{k-2}^{\oplus}(E)$,}
\]
\item $\mathbf{D_V4}$: the moments of ${\rm div} \,\mathbf{v}$ 
\[
\int_E ({\rm div} \,\mathbf{v}) \, q_{k-1} \, {\rm d}E \qquad \text{for all $q_{k-1} \in \Pk_{k-1}(E) / \R$}
\] 
\end{itemize}

\end{proposition}

\begin{figure}[!h]
\center{
\includegraphics[scale=0.20]{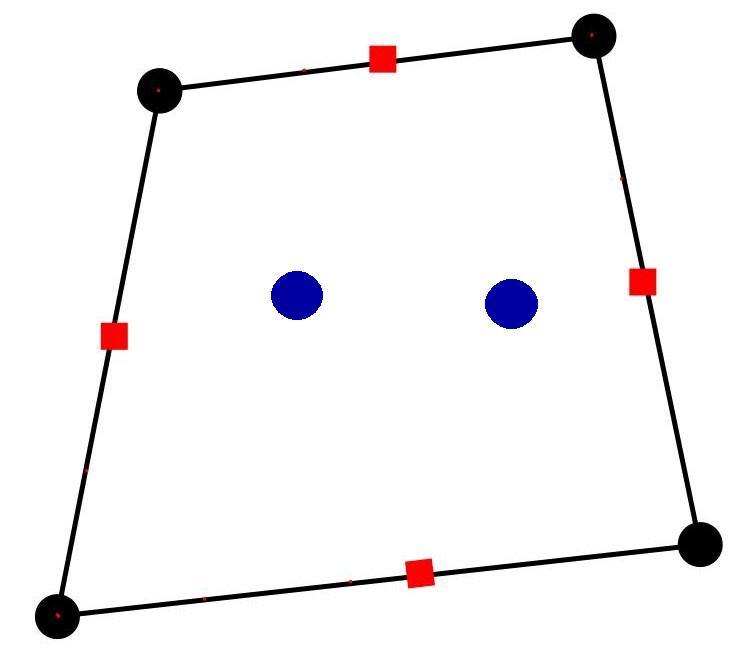} \qquad \qquad
\includegraphics[scale=0.20]{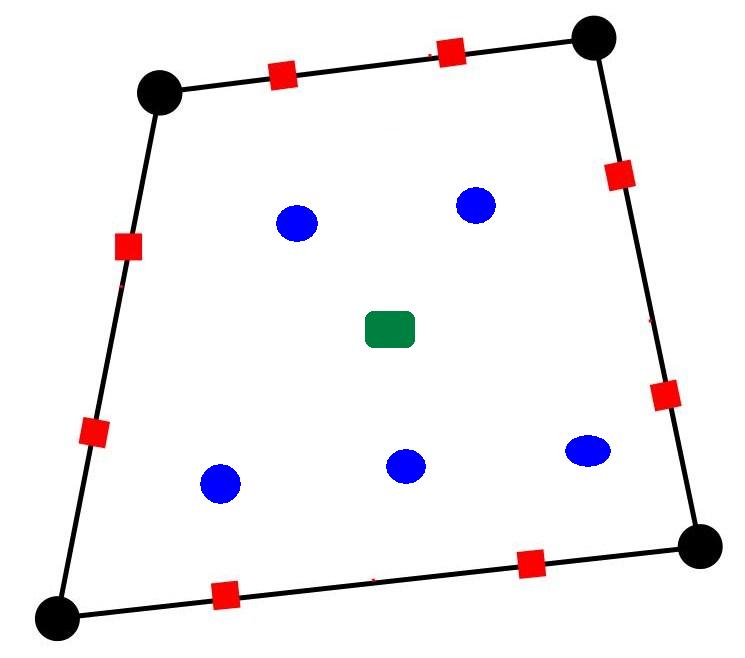}
\caption{Degrees of freedom for $k=2$, $k=3$. We denote $\mathbf{D_V1}$ with black dots, $\mathbf{D_V2}$ with red squares, $\mathbf{D_V3}$ with green rectangles, $\mathbf{D_V4}$ with blue dots inside the element.}
\label{fig:dofsloc}
}
\end{figure}

The proof of the following result can be found in \cite{Stokes:divfree} for $\PN$ and in \cite{preprintdarcy} for the remaining projectors.
\begin{proposition}[Projections and Computability]
\label{prp:projections}
The DoFs $\mathbf{D_V}$ allow us to compute exactly 
\[
\PN \colon \VV_h^E \to [\Pk_k(E)]^2, \qquad
\P0 \colon \VV_h^E \to [\Pk_k(E)]^2, \qquad
\PP0 \colon \Gr(\VV_h^E) \to [\Pk_{k-1}(E)]^{2 \times 2},
\]
in the sense that, given any $\vv_h \in \VV_h^E$, we are able to compute the polynomials
$\PN \vv_h$, $\P0 \vv_h$ and $\PP0\nabla\vv_h$ only using, as unique information, the degree of freedom values $\mathbf{D_V}$ of $\vv_h$.
\end{proposition}

\begin{remark}
\label{rm1}
Using the enhanced space $\VV_h^E$ and following the same ideas of \cite{Stokes:divfree,preprintdarcy}, it is possible to improve the results of Proposition \ref{prp:projections} and compute exactly also the following higher order projections
\[
\Pi^{\nabla, E}_{k+2} \colon \VV_h^E \to [\Pk_{k+2}(E)]^2, \qquad
\boldsymbol{\Pi}^{0, E}_{k+1} \colon \Gr(\VV_h^E) \to [\Pk_{k+1}(E)]^{2 \times 2}.
\] 
Moreover, given any polynomial $q_n$ of arbitrary degree n and any $\vv \in \VV_h^E$, an integration by parts shows that we can compute the moment
\[
\int_E \nabla q_n \cdot \vv \, {\rm d}E. 
\] 
\end{remark}
For what concerns the pressures we take the standard finite dimensional space
\begin{equation}
\label{eq:Q_h^E}
Q_h^E := \Pk_{k-1}(E) 
\end{equation}
having dimension
\begin{equation*}
\dim(Q_h^E) = \dim(\Pk_{k-1}(E))  = \frac{(k+1)k}{2}.
\end{equation*}
The corresponding degrees of freedom are chosen defining for each $q\in Q_h^E$ the following linear operators $\mathbf{D_Q}$:
\begin{itemize}
\item $\mathbf{D_Q}$: the moments up to order $k-1$ of $q$, i.e.
\[
\int_E q \, p_{k-1} \, {\rm d}E \qquad \text{for all $p_{k-1} \in \Pk_{k-1}(E)$.}
\]
\end{itemize}
Finally we define the global virtual element spaces as
\begin{equation}
\label{eq:V_h}
\mathbf{V}_h := \{ \mathbf{v} \in [H^1_0(\Omega)]^2  \quad \text{s.t} \quad \mathbf{v}_{|E} \in \mathbf{V}_h^E  \quad \text{for all $E \in \Omega_h$} \}
\end{equation} 
and
\begin{equation}
\label{eq:Q_h}
Q_h := \{ q \in L_0^2(\Omega) \quad \text{s.t.} \quad q_{|E} \in  Q_h^E \quad \text{for all $E \in \Omega_h$}\},
\end{equation}
with the obvious associated sets of global degrees of freedom. A simple computation shows that:
\begin{equation*}
\dim(\mathbf{V}_h) = n_P \left( \frac{(k+1)k}{2} -1  +  \frac{(k-1)(k-2)}{2} \right) 
+ 2(n_V + (k-1) n_e)
\end{equation*}
and
\begin{equation*}
\dim(Q_h) = n_P \frac{(k+1)k}{2} - 1 ,
\end{equation*}
where $n_P$ is the number of elements, $n_e$, $n_V$ is the number of  internal edges and vertexes in $\Omega_h$.
As observed in \cite{Stokes:divfree}, we remark that
\begin{equation}\label{eq:divfree}
{\rm div}\, \mathbf{V}_h\subseteq Q_h .
\end{equation}
%
%

%

\subsection{Discrete bilinear forms and load term approximation}
\label{sub:3.2}

The next step in the construction of our method is to define a discrete version of the bilinear forms $a(\cdot, \cdot)$ and $b(\cdot, \cdot)$ given in \eqref{eq:forma a} and \eqref{eq:forma b} and trilinear form $c(\cdot; \cdot, \cdot)$ in \eqref{eq:forma c}.
Here and in the rest of the paper the symbol $C$ will indicate a generic positive quantity, independent of the mesh size (and of $\nu$), but may depend on $\Omega$ and on the polynomial degree $k$. Furthermore, $C$ may vary at each occurrence.
First of all we decompose into local contributions the bilinear forms $a(\cdot, \cdot)$, $b(\cdot, \cdot)$, the trilinear form $c(\cdot; \cdot, \cdot)$ and the norms $\|\cdot\|_{\mathbf{V}}$, $\|\cdot \|_Q$ by defining 
\begin{equation*}
a (\mathbf{u},  \mathbf{v}) =: \sum_{E \in \Omega_h} a^E (\mathbf{u},  \mathbf{v}) \qquad \text{for all $\mathbf{u},  \mathbf{v} \in \mathbf{V}$}
\end{equation*}
\begin{equation*}
b (\mathbf{v},  q) =: \sum_{E \in \Omega_h} b^E (\mathbf{v},  q) \qquad \text{for all $\mathbf{v} \in \mathbf{V}$ and $q \in Q$,}
\end{equation*}
\begin{equation*}
c(\ww; \, \uu, \vv) =: \sum_{E \in \Omega_h} c^E(\ww; \, \uu, \vv) \qquad \text{for all $\ww, \uu, \vv \in \VV$.}
\end{equation*}
and 
\begin{equation*}
\|\mathbf{v}\|_{\mathbf{V}} =: \left(\sum_{E \in \Omega_h} \|\mathbf{v}\|^2_{\mathbf{V}, E}\right)^{1/2} \quad \text{for all $\mathbf{v} \in \mathbf{V}$,} \qquad \|q\|_Q =: \left(\sum_{E \in \Omega_h} \|q\|^2_{Q, E}\right)^{1/2} \quad \text{for all $q \in Q$.}
\end{equation*}
For what concerns $b(\cdot, \cdot)$, we simply  set 
\begin{equation}\label{bhform}
b(\mathbf{v}, q) = \sum_{E \in \Omega_h} b^E(\mathbf{v}, q) = \sum_{E \in \Omega_h}  \int_E {\rm div} \, \mathbf{v} \, q \,{\rm d}E \qquad \text{for all $\mathbf{v} \in \mathbf{V}_h$, $q \in Q_h$},
\end{equation}
i.e. as noticed in \cite{Stokes:divfree} we do not introduce any approximation of the bilinear form. We notice that~\eqref{bhform}  
is computable from the degrees of freedom $\mathbf{D_V1}$, $\mathbf{D_V2}$ and $\mathbf{D_V4}$, since $q$ is polynomial in each element $E \in \Omega_h$. 
We now define discrete versions of the forms $a(\cdot, \cdot)$ (cf.~\eqref{eq:forma a}) and  $c(\cdot; \,\cdot, \cdot)$ (cf.~\eqref{eq:forma c}) that need  to be dealt with in a more careful way.
First of all, we note that for an arbitrary pair $(\mathbf{u},\mathbf{v} )\in \mathbf{V}_h^E \times \mathbf{V}_h^E $, the quantity $a_h^E(\mathbf{w}, \mathbf{v})$ is not computable.  Therefore, following a standard procedure in the VEM framework, we define a computable discrete local bilinear form
\begin{equation}
\label{eq:a_h^E} 
a_h^E(\cdot, \cdot) \colon \mathbf{V}_h^E \times \mathbf{V}_h^E \to \R
\end{equation}
approximating the continuous form $a^E(\cdot, \cdot)$, and defined by
\begin{equation}
\label{eq:a_h^E def}
a_h^E(\mathbf{u}, \mathbf{v}) := a^E \left(\PN \mathbf{u}, \, \PN \mathbf{v} \right) + \mathcal{S}^K \left((I - \PN) \mathbf{u}, \, (I -\PN) \mathbf{v} \right)
\end{equation}
for all $\mathbf{u}_h, \mathbf{v}_h \in \mathbf{V}_h^E$, where the (symmetric) stabilizing bilinear form $\mathcal{S}^E \colon \mathbf{V}_h^E \times \mathbf{V}_h^E \to \R$, satisfies (see Remark \ref{rm:stabilizzazione})
\begin{equation}
\label{eq:S^E}
\alpha_* a^E(\mathbf{v}, \mathbf{v}) \leq  \mathcal{S}^E(\mathbf{v}, \mathbf{v}) \leq \alpha^* a^E(\mathbf{v}, \mathbf{v}) \qquad \text{for all $\mathbf{v} \in \mathbf{V}_h$ such that ${\Pi}_{k}^{\nabla ,E} \mathbf{v}= \mathbf{0}$}
\end{equation}
with $\alpha_*$ and $\alpha^*$  positive  constants independent of the element $E$.
It is straightforward to check that Definition~\eqref{eq:Pn_k^E} and properties~\eqref{eq:S^E} imply 
\begin{itemize}
\item $\mathbf{k}$\textbf{-consistency}: for all $\mathbf{q}_k \in [\Pk_k(E)]^2$ and $\mathbf{v} \in \mathbf{V}_h^K$
\begin{equation}\label{eq:consist}
a_h^E(\mathbf{q}_k, \mathbf{v}) = a^E( \mathbf{q}_k, \mathbf{v});
\end{equation}
\item \textbf{stability}:  there exist  two positive constants $\alpha_*$ and $\alpha^*$, independent of $h$ and $E$, such that, for all $\mathbf{v} \in \mathbf{V}_h^E$, it holds
\begin{equation}\label{eq:stabk}
\alpha_* a^E(\mathbf{v}, \mathbf{v}) \leq a_h^E(\mathbf{v}, \mathbf{v}) \leq \alpha^* a^E(\mathbf{v}, \mathbf{v}).
\end{equation}
\end{itemize}

\begin{remark}
\label{rm:stabilizzazione}
Condition \eqref{eq:S^E}  essentially requires that the stabilizing term $\mathcal{S}^E(\mathbf{v}_h, \mathbf{v}_h)$ scales as $a^E(\mathbf{v}_h, \mathbf{v}_h)$. For instance, following the most standard VEM choice (cf. \cite{volley,VEM-hitchhikers,2016stability}), denoting with $\vec{\mathbf{u}}_h$, $\vec{\mathbf{v}}_h \in \R^{N_{DoFs, E}}$
the vectors containing the values of the $N_{DoFs, E}$ degrees of freedom associated to $\mathbf{u}_h, \mathbf{v}_h \in \mathbf{V}_h^E$, we set
\[
\mathcal{S}^E (\mathbf{u}_h, \mathbf{v}_h) =  \alpha^E \,  \vec{\mathbf{u}}_h^T \vec{\mathbf{v}}_h ,
\] 
where $\alpha^E$ is a suitable positive constant. For example, in the numerical tests presented in Section \ref{sec:5}, we have chosen $\alpha^E$ as the mean value of the non-zero eigenvalues of the matrix stemming from the  
term $a^E \left(\PN \mathbf{u}_h,\,  \PN \mathbf{v}_h \right) $ in  \eqref{eq:a_h^E def}.
\end{remark}

Finally we define the global approximated bilinear form $a_h(\cdot, \cdot) \colon \mathbf{V}_h \times \mathbf{V}_h \to \R$ by simply summing the local contributions:
\begin{equation}
\label{eq:a_h}
a_h(\mathbf{u}_h, \mathbf{v}_h) := \sum_{E \in \Omega_h}  a_h^E(\mathbf{u}_h, \mathbf{v}_h) \qquad \text{for all $\mathbf{u}_h, \mathbf{v}_h \in \mathbf{V}_h$.}
\end{equation}

For what concerns the approximation of the local trialinear form $c^E(\cdot; \, \cdot, \cdot)$, we set
\begin{equation}
\label{eq:c_h^E}
c_h^E(\ww_h; \, \uu_h, \vv_h) := \int_E \left[ \left(\PP0 \, \Gr \uu_h  \right)  \left(\P0 \ww_h \right) \right] \cdot \P0 \vv_h \, {\rm d}E 
\qquad \text{for all $\ww_h, \uu_h, \vv_h \in \VV_h$}
\end{equation}
and note that all quantities in the previous formula are computable, in the sense of Proposition \ref{prp:projections}. 
As usual we define the global approximated trilinear form by adding the local contributions:
\begin{equation}
\label{eq:c_h}
c_h(\ww_h; \, \uu_h, \vv_h) := \sum_{E \in \Omega_h}  c_h^E(\ww_h; \, \uu_h, \vv_h), \qquad \text{for all $\ww_h, \uu_h, \vv_h \in \VV_h$.}
\end{equation}
We first notice that the form $c_h(\cdot; \, \cdot, \cdot)$ is immediately extendable to the whole $\VV$ (simply apply the same definition for any $\ww, \uu, \vv \in \VV$) . Moreover, we now show that it is continuous on $\VV$, uniformly in $h$.
\begin{proposition}
\label{prp:continuity-ch}
Let 
\begin{equation}
\label{eq:CC}
\widehat{C}_h := \sup_{\ww, \uu, \vv \in \VV} \frac{|c_h(\ww; \, \uu, \vv)|}{\|\ww\|_{\VV} \|\uu\|_{\VV} \|\vv\|_{\VV}}.
\end{equation}
Then $\widehat{C}_h$ is uniformly bounded, i.e. the trilinear form $c_h(\cdot; \, \cdot, \cdot)$ is uniformly continuous with respect to $h$.
\end{proposition}

\begin{proof}
By a direct computation it holds
\begin{equation}
\label{eq:continuity-ch1}
\begin{split} 
c_h(\ww; \, \uu, \vv) & = \sum_{E \in \Omega_h} c_h^E(\ww; \, \uu, \vv) = 
\sum_{E \in \Omega_h} \int_E \left[ \left(\PP0 \, \Gr \uu  \right)  \left(\P0 \ww \right) \right] \cdot \P0 \vv \, {\rm d}E \\
& \leq \sum_{i,j=1}^2 \, \sum_{E \in \Omega_h} 
\left\|\P0 \, \frac{\partial \uu_i}{\partial x_j} \right\|_{0,E}
\left\| \P0 \ww_j\right\|_{L^4(E)}
\left\| \P0 \vv_i\right\|_{L^4(E)}
\end{split}
\end{equation}
where the last inequality follows by using H\"older inequality.
Let us analyse each term in the right hand side of \eqref{eq:continuity-ch1}. Employing the continuity of the projection $\P0$ with respect the $L^2$-norm we easily get
\begin{equation}
\label{eq:continuity-ch2}
\left\|\P0 \, \frac{\partial \uu_i}{\partial x_j} \right\|_{0,E} \leq \left\| \frac{\partial \uu_i}{\partial x_j} \right\|_{0,E}.
\end{equation}
For what concerns the second term (and analogously for the third one) we get
\begin{equation}
\label{eq:continuity-ch3}
\begin{aligned}
\left\| \P0 \ww_j\right\|_{L^4(E)} & \leq C h_E^{-\frac{1}{2}} \left\| \P0 \ww_j\right\|_{0,E}   &\text{(inverse estimate for polynomials)} \\
&  \leq C h_E^{-\frac{1}{2}} \left\| \ww_j\right\|_{0,E}   &\text{(continuity of $\P0$ with respect $\|\cdot\|_{0,E}$ )} \\
&  \leq C h_E^{-\frac{1}{2}} \,  \|1\|_{L^4(E)} \, \left\| \ww_j\right\|_{L^4(E)}   &\text{(H\"older inequality )} \\
&  \leq C h_E^{-\frac{1}{2}} \,  (h_E^2)^{\frac{1}{4}} \, \left\| \ww_j\right\|_{L^4(E)}   &\text{(definition of $h_E$)} \\
&\leq C \, \left\| \ww_j\right\|_{L^4(E)}. 
\end{aligned}
\end{equation}
Collecting \eqref{eq:continuity-ch2} and \eqref{eq:continuity-ch3} in \eqref{eq:continuity-ch1} we obtain
\begin{equation}
\label{eq:continuity-ch4}
c_h(\ww; \, \uu, \vv) \leq C  \sum_{i,j=1}^2 \, \sum_{E \in \Omega_h} \left\| \frac{\partial \uu_i}{\partial x_j} \right\|_{0,E} \, \left\| \ww_j\right\|_{L^4(E)} \, \left\| \vv_i\right\|_{L^4(E)}.
\end{equation} 
Now applying H\"older inequality (for sequences) we get
\begin{equation}
\label{eq:continuity-ch5}
\begin{split}
c_h(\ww; \, \uu, \vv) & \leq C  \sum_{i,j=1}^2 \, \left( \sum_{E \in \Omega_h} \left\| \frac{\partial \uu_i}{\partial x_j} \right\|^2_{0,E} \right)^{\frac{1}{2}} \, \left(\sum_{E \in \Omega_h} \left\| \ww_j\right\|^4_{L^4(E)} \right)^{\frac{1}{4}} \, \left( \sum_{E \in \Omega_h}\left\| \vv_i\right\|^4_{L^4(E)} \right)^{\frac{1}{4}} \\
& \leq C  \sum_{i,j=1}^2 \left\| \frac{\partial \uu_i}{\partial x_j} \right\|_{0} \,  \left\| \ww_j\right\|_{L^4(\Omega)} \,  \left\| \vv_i\right\|_{L^4(\Omega)}.
\end{split}
\end{equation}
Finally, since $H^1(\Omega) \subset L^4(\Omega)$, by Sobolev embedding it holds 
\[
c_h(\ww; \, \uu, \vv) \leq \widehat{C}_h \, \|\uu\|_{\VV} \|\ww\|_{\VV} \|\vv\|_{\VV},
\]
where the constant $\widehat{C}_h$ does not depend on $h$.
\end{proof}
We can also define the local discrete skew-symmetric trilinear form $\widetilde{c}_h^E(\cdot; \, \cdot, \cdot)\colon  \VV \times \VV \times \VV \to \R$ simply setting
\begin{equation}
\label{eq:ctilde_h^E}
\widetilde{c}_h^E(\ww; \, \uu, \vv) := \frac{1}{2}c_h^E(\ww; \, \uu, \vv) - \frac{1}{2}c_h^E(\ww; \, \vv, \uu)
\qquad \text{for all $\ww, \uu, \vv \in \VV$}
\end{equation}
with obvious global extension
\begin{equation}
\label{eq:ctilde_h}
\widetilde{c}_h(\ww; \, \uu, \vv) := \sum_{E \in \Omega_h}  \widetilde{c}_h^E(\ww; \, \uu, \vv), \qquad \text{for all $\ww, \uu, \vv \in \VV$,}
\end{equation}
that is (obviously) still continuous and computable.

The last step consists in constructing a computable approximation of the right-hand side $(\mathbf{f}, \, \mathbf{v})$ in \eqref{eq:ns variazionale}.  We define the approximated load term $\mathbf{f}_h$ as 
\begin{equation}
\label{eq:f_h}
\mathbf{f}_h := \Pi_{k}^{0,E} \mathbf{f} \qquad \text{for all $E \in \Omega_h$,}
\end{equation}
and consider:
\begin{equation}
\label{eq:right}
(\mathbf{f}_h, \mathbf{v}_h)  = \sum_{E \in \Omega_h} \int_E \mathbf{f}_h \cdot \mathbf{v}_h \, {\rm d}E = \sum_{E \in \Omega_h} \int_E \Pi_{k}^{0,E} \mathbf{f} \cdot \mathbf{v}_h \, {\rm d}E = \sum_{E \in \Omega_h} \int_E \mathbf{f} \cdot \Pi_{k}^{0,E}  \mathbf{v}_h \, {\rm d}E.
\end{equation}
We observe that \eqref{eq:right} can be exactly computed from $\mathbf{D_V}$ for all $\mathbf{v}_h \in \mathbf{V}_h$ (see Proposition \ref{prp:projections}). 

\subsection{The discrete problem}\label{sec:discrete}
\label{sub:3.3}
We are now ready to state the proposed discrete problem. Referring to~\eqref{eq:V_h}, \eqref{eq:Q_h},  ~\eqref{eq:a_h},  ~\eqref{eq:ctilde_h} and ~\eqref{bhform}, we consider the \textbf{virtual element problem}:
\begin{equation}
\label{eq:ns virtual}
\left\{
\begin{aligned}
& \text{find $(\uu_h, p_h) \in \VV_h \times Q_h$, such that} \\
& \nu \, a_h(\uu_h, \vv_h) + \widetilde{c}_h(\uu_h; \,  \uu_h, \vv_h) + b(\vv_h, p_h) = (\ff_h, \vv_h) \qquad & \text{for all $\vv_h \in \VV_h$,} \\
&  b(\uu_h, q_h) = 0 \qquad & \text{for all $q_h \in Q_h$.}
\end{aligned}
\right.
\end{equation}

We point out that the symmetry of $a_h(\cdot, \cdot)$ together with \eqref{eq:stabk}  easily implies that $a_h(\cdot, \cdot)$  is  continuous and coercive with respect to the  $\VV$-norm. 
Moreover, as a direct consequence of Proposition 4.3 in \cite{Stokes:divfree}, we have the following stability result.
\begin{proposition}
\label{thm2} Given the discrete spaces
$\mathbf{V}_h$ and $Q_h$ defined in~\eqref{eq:V_h} and~\eqref{eq:Q_h}, there exists a positive $\widehat{\beta}$, independent of $h$, such that:
\begin{equation}
\label{eq:inf-sup discreta}
\sup_{\mathbf{v}_h \in \mathbf{V}_h \, \mathbf{v}_h \neq \mathbf{0}} \frac{b(\mathbf{v}_h, q_h)}{ \|\mathbf{v}_h\|_{\mathbf{V}}} \geq \widehat{\beta} \|q_h\|_Q \qquad \text{for all $q_h \in Q_h$.}
\end{equation}
\end{proposition}
In particular, the inf-sup condition of Proposition~\ref{thm2}, along with property~\eqref{eq:divfree}, implies that:
\begin{equation*}
{\rm div}\, \mathbf{V}_h = Q_h .
\end{equation*}
The well-posedness of virtual problem \eqref{eq:ns virtual} is a consequence of the coercivity property of $a_h(\cdot,\cdot)$, the skew-symmetry of $\widetilde{c}_h(\cdot; \cdot,\cdot)$ and the inf-sup condition \eqref{eq:inf-sup discreta}. We have
\begin{theorem}
Assuming that
\begin{equation}
\label{eq:ns virtual condition}
\gamma_h := \frac{\widehat{C}_h \, \|\ff_h\|_{-1}}{\alpha_*^2 \, \nu^2} \le r < 1
\end{equation}
Problem \eqref{eq:ns virtual} has a unique solution $(\uu_h, p_h) \in \VV_h \times Q_h$ such that
\begin{equation}
\label{eq:solution virtual estimates}
\| \uu_h\|_{\VV} \leq \frac{\| \ff_h\|_{H^{-1}}}{\alpha_* \, \nu}.
\end{equation}
\end{theorem}
Moreover, as observed in \cite{Stokes:divfree}, introducing the discrete kernel
\begin{equation*}
\mathbf{Z}_h := \{ \mathbf{v}_h \in \mathbf{V}_h \quad \text{s.t.} \quad b(\mathbf{v}_h, q_h) = 0 \quad \text{for all $q_h \in Q_h$}\},
\end{equation*}
recalling \eqref{eq:divfree} it follows
\begin{equation}
\label{kernincl}
\mathbf{Z}_h \subseteq \mathbf{Z} .
\end{equation}
Problem \eqref{eq:ns virtual} can be also formulated in the equivalent kernel form
\begin{equation}
\label{eq:nsvirtual ker}
\left\{
\begin{aligned}
& \text{find $\uu_h \in \ZZ_h$, such that} \\
& \nu \, a_h(\uu_h, \vv_h) + \widetilde{c}_h(\uu_h; \, \uu_h, \vv_h) = (\ff_h, \vv_h) \qquad & \text{for all $\vv_h \in \ZZ_h$.} 
\end{aligned}
\right.
\end{equation}

\begin{remark}\label{rem:non-skew}
An alternative choice for the discretization \eqref{eq:ns virtual} is to substitute the skew-simmetric form $\widetilde{c}_h(\cdot;\cdot,\cdot)$ with 
${c}_h(\cdot;\cdot,\cdot)$. With that choice, a theoretical analysis can be developed using the guidelines in \cite{Maday_Quarteroni} in connection 
with the same tools and ideas of Section \ref{sec:4}. We here prefer to consider the choice \eqref{eq:ns virtual}, that allows for a more direct stability argument. Nevertheless, in the numerical tests of Section \ref{sec:5} we will investigate both possibilities.
\end{remark}

\begin{remark}\label{rem:Sto-Dar} 
An additional interesting consequence of property \eqref{kernincl} is that, following \cite{Stokes:divfree,preprintdarcy}, the proposed virtual elements can accommodate both the Stokes (or Navier-Stokes) and the Darcy problems simultaneously. Indeed, due to property \eqref{kernincl}, the proposed velocity-pressure couple turns out to be stable not only for the Stokes problem, but also for the Darcy problem. This yields an interesting advantage in complex flow problems where both equations are present: the same spaces can be used in the whole computational domain. As a consequence, the implementation of the method and the enforcement of the interface conditions are greatly simplified (see also Section \ref{test6}).

\end{remark}

\section{Theoretical analysis}
\label{sec:4}

\subsection{Interpolation estimates}
\label{sub:4.1}
In this section we prove that the following interpolation estimate holds for the enhanced space $\VV_h$. Since the proof is quite involved, we divide it in three steps.
\begin{theorem}
\label{thm:interpolante}
Let $\vv \in H^{s+1}(\Omega) \cap \VV$, for $0<s \le k$. Then there exists $\vv_I \in \VV_h$ such that
\[
\| \vv - \vv_I \|_0 + h \, \| \vv - \vv_I \|_{\VV} \leq C \, h^{s+1} \, | \vv |_{s+1},
\]
where the constant $C$ depends only on the degree $k$ and the shape regularity constants $\rho,c$ (see assumptions $\mathbf{(A1)}$ and $\mathbf{(A2)}$ of Section \ref{sub:3.1}).
\end{theorem}
\begin{proof} 
\textit{Step 1.} 
Let $\ww_I$ the approximant function of $\vv$ in the space $\mathbf{W}_h$ (cf. \eqref{eq:W_h} and Proposition 4.2 in  \cite{Stokes:divfree}) then it holds that
\begin{equation}
\label{eq:interpolata old}
\| \vv - \ww_I \|_0 + h \, \| \vv - \ww_I \|_{\VV} \leq C \, h^{s+1} \, | \vv |_{s+1}.
\end{equation}
Now let $\vv_I \in \VV_h$ be the interpolant of $\ww_I$ in the sense of the DoFs $\mathbf{D_V}$, so that
\begin{equation}\label{Dv-eq-I}
\mathbf{D_V}(\vv_I)= \mathbf{D_V}(\ww_I).
\end{equation}
Let us define $\te := \vv_I - \ww_I$, then for every element $E \in \Omega_h$ the following facts hold.
\begin{itemize}
\item Since $\vv_I$ and $\ww_I$ are polynomials of degree $k$ on $\partial E$, by definition of  $\mathbf{D_V1}$ and $\mathbf{D_V2}$, we have
\begin{equation}
\label{eq:inter1}
\te = \mathbf{0} \qquad \text{on $\partial E$.}
\end{equation}
\item Since $\dd \, \vv_I$ and $\dd \, \ww_I$ are polynomials of degree $k-1$ in $E$, by definition of  $\mathbf{D_V4}$ and homogeneous boundary data \eqref{eq:inter1}, we get
\begin{equation}
\label{eq:inter2}
\dd \, \te = 0 \qquad \text{in $E$.}
\end{equation}
\item Let $d^E(\cdot, \, \cdot) \colon H^1_0(E) \times \mathcal{G}_k^{\oplus}(E) \to \R$ given by
\[
d^E(\vv, \, \ggq) =\int_E \vv \cdot \ggq \,{\rm d}E \qquad \text{for all $\vv \in H^1_0(E)$, and $\ggq \in \GGq$} .
\]
Then by definition of  $\mathbf{D_V3}$, we infer
\begin{equation}\label{added:star1}
d^E(\te, \, \ggp) = 0\qquad \text{for all $\ggp \in \GGp$}.
\end{equation}
Now we recall that, for any $\vv_h \in \VV_h$, the quantity $\PN \vv_h$ depends only on the values of $\mathbf{D_V}(\vv_h)$, see Proposition \ref{prp:projections}. Therefore, using \eqref{Dv-eq-I}, we have that $\PN\vv_I = \PN\ww_I$. As a consequence, by definition of $\VV_h^E$ it holds
\begin{equation}\label{added:star2}
d^E(\te, \, {\mathbf{g}}^{\perp}) =  \int_E \left( \PN \vv_I - \ww_I \right) \cdot {\mathbf{g}}^{\perp} \, {\rm d}E =  \int_E \left( \PN \ww_I - \ww_I \right) \cdot {\mathbf{g}}^{\perp} \, {\rm d}E
\end{equation}
for all $\mathbf{g}^{\perp} \in \GGq \setminus \GGp$. Thus, by \eqref{added:star1} and \eqref{added:star2},
\begin{equation}
\label{eq:inter3}
d^E(\te, \, \ggq) = (\xx, \, \ggq) \qquad \text{for all $\ggq \in \GGq$} 
\end{equation}
where 

\begin{equation}\label{eq:chi}
\mbox{
$\xx$ is the $L^2$-projection of $\left( \PN \ww_I - \ww_I \right)$ onto $\GGq \setminus \GGp$.}
\end{equation}

\item By definition of $\WW_h^E$ and $\VV_h^E$ there exist $\widehat{s} \in L_0^2(E)$ and $\widehat{\mathbf{g}} \in \GGq$ such that
\begin{equation}
\label{eq:inter4}
a^E(\te, \vv) +b^E(\vv, \widehat{s}) + d^E(\vv, \widehat{\mathbf{g}}) = 0 \qquad \text{for all $\vv \in H^1_0(E)$.} 
\end{equation}
\end{itemize}
Collecting \eqref{eq:inter1}, \eqref{eq:inter2}, \eqref{eq:inter3}, \eqref{eq:inter4} it follows that $(\te, \, \widehat{s}, \, \widehat{\mathbf{g}})$ solves the problem
\begin{equation}
\label{eq:pbinter}
\left\{
\begin{aligned}
& \text{Find $(\te, \, \widehat{s}, \, \widehat{\mathbf{g}}) \in [H^1_0(E)]^2 \times L^2_0(E) \times \GGq$, such that} \\
& a^E(\te, \fgl) +  b^E(\fgl, \widehat{s}) + d^E(\fgl, \widehat{\mathbf{g}}) = 0 \qquad & \text{for all $\fgl \in [H^1_0(E)]^2$,} \\
& b^E(\te, q) = 0 \qquad & \text{for all $q \in L^2_0(E)$,} \\
& d^E(\te, \mathbf{h}) = (\xx, \mathbf{h}) \qquad  & \text{for all $\mathbf{h} \in \GGq$.}
\end{aligned}
\right.
\end{equation}
\textit{Step 2.}
We now analyse the well-posedness of Problem \eqref{eq:pbinter}. We consider $[H^1_0(E)]^2$ and $L^2(E)$ endowed with the $H^1$ and the $L^2$-norm, respectively, and $\GGq$ endowed with the scaled norm
\[
\|\mathbf{h}\|_{\GGq} := h_E \, \|\mathbf{h}\|_{0,E} \qquad \text{for all $\mathbf{h} \in \GGq$.}
\]
Then for all $\fgl \in [H^1_0(E)]^2$ and $\mathbf{h} \in \GGq$
\begin{equation}
\label{eq:dcon}
d^E(\fgl, \mathbf{h}) = \int_E \fgl \cdot \mathbf{h} \, {\rm d}E \leq \|{\fgl}\|_{0,E} \|\mathbf{h}\|_{0,E} \leq c_{{\rm cont}} \,  |\fgl|_{1,E} \, h^E \|\mathbf{h}\|_{0,E} ,
\end{equation}
where the last inequality follows by a scaled Poincar\'e inequality.
Therefore all the involved bilinear forms are continuous. By the theory of problems in mixed form \cite{BoffiBrezziFortin}, due to the coercivity of $a^E(\cdot,\cdot)$ the well-posedness of problem \eqref{eq:pbinter} will follow if we show an inf-sup condition for the form
$$
b^E(\cdot,\cdot) \! + \! d^E(\cdot,\cdot) \ \colon \
[H^1_0(E)]^2 \times \big(L^2_0(E) \times \GGq) \to {\mathbb R} .
$$
In other words, for all $(q, \mathbf{h}) \in L^2_0(E) \times \GGq$ we have to find $\fg \in H^1_0(E)$ such that
\begin{equation}
\label{eq:is3inter}
\left\{
\begin{aligned}
& |\fg|_{1,E} \leq  b_0\, (\|q\|_{0,E} + \|\mathbf{h}\|_{\GGq} ) \\
& b^E(\fg, q) +  d^E(\fg, \mathbf{h}) \geq c_0 \, ( \|q\|_{0,E} + \|\mathbf{h}\|_{\GGq})^2
\end{aligned}
\right.
\end{equation}
for suitable uniform positive constants  $b_0$, $c_0$.
It is well known (see \cite{BoffiBrezziFortin}) that for all $q \in L^2_0(E)$ there exists $\fg_1 \in [H^1_0(E)]^2$ such that
\begin{equation}
\label{eq:is1inter}
\left\{
\begin{aligned}
& |\fg_1|_{1,E} \leq b_1 \, \|q\|_{0,E} \\
& b^E(\fg_1, q) \geq c_1 \, \|q\|^2_{0,E}.
\end{aligned}
\right.
\end{equation}
Now let $T_E \subset E$ be an equilateral triangle inscribed in the ball $B_E$ (cf. assumption $\mathbf{(A1)}$). Then for all polynomial $p \in \Pk_k(E)$, it holds $\|p\|_{0,E} \leq C \|p\|_{0, T_E}$ for a suitable uniform constant $C$. Let $\mathbf{h} \in \GGq$ and we define
\[
q := \rr (\mathbf{h}) \qquad \text{and} \qquad \fg_2 := h_E^4 \, \cc(b q)
\]
where $b \in \Pk_3(T_E)$ denotes the standard cubic bubble in $T_E$ with unitary maximum value. Therefore, we get
\begin{equation}
\label{eq:is21inter}
\begin{split}
d^E(\fg_2, \mathbf{h}) & = h_E^4 \, \int_E \cc(bq) \cdot \mathbf{h} \, {\rm d}E
= h_E^4 \, \int_E b q \,   \rr(\mathbf{h}) \, {\rm d}E 
= h_E^4 \, \int_E b \,   \rr(\mathbf{h})^2 \, {\rm d}E \\
& \geq  h_E^4 \,\|\rr(\mathbf{h})\|_{0,E}^2 \geq C h_E^4 \, \|\rr(\mathbf{h})\|_{0,T_E}^2.
\end{split}
\end{equation}
Since $\rr \colon \mathcal{G}_k^{\oplus}(T_E)\to \Pk_{k-1}(T_E)$ is an isomorphism (see \cite{supermisti}), a scaling argument for polynomials on the triangle $T_E$ yields $\|\rr(\mathbf{h})\|_{0,T_E} \geq h_E^{-1} \|\mathbf{h}\|_{0,T_E}$. Thus using  \eqref{eq:is21inter} we find
\begin{equation}
\label{eq:is22inter}
d^E(\fg_2, \mathbf{h}) \geq C \,h_E^4 \, h_E^{-2} \, \|\mathbf{h}\|_{0,T_E}^2 \geq C \, h_E^{2} \, \|\mathbf{h}\|_{0,E}^2 = C \, \|\mathbf{h}\|_{\GGq}^2.
\end{equation}
Moreover using an inverse estimate for the polynomials $bq$ and $\mathbf{h}$
\begin{equation}
\label{eq:is23inter}
\begin{split}
|\fg_2|_{1,E} &= h_E^4 \,|\cc(b q)|_{1,E} \leq C h_E^4 \, \, h_E^{-2} \|b q\|_{0,E} \leq C \, h_E^{2} \|q\|_{0,E} \\
& = C  \, h_E^{2} \|\rr(\mathbf{h})\|_{0,E} \leq C \, h_E \,  \|\mathbf{h}\|_{0,E} = C  \|\mathbf{h}\|_{\GGq(E)}.
\end{split}
\end{equation}
Therefore by \eqref{eq:is22inter} and \eqref{eq:is23inter} for all $\mathbf{h} \in \GGq$ we find $\fg_2 \in H^1_0(E)$ such that
\begin{equation}
\label{eq:is2inter}
\left\{
\begin{aligned}
& |\fg_2|_{1,E} \leq b_2 \, \|\mathbf{h}\|_{\GGq} \\
& d^E(\fg_2, \mathbf{h}) \geq c_2 \, \|\mathbf{h}\|^2_{\GGq}.
\end{aligned}
\right.
\end{equation}
Recalling \eqref{eq:is3inter}, let us set $\fg := \fg_1 + \xi \,\fg_2$ (cf. \eqref{eq:is1inter} and \eqref{eq:is2inter}) where $\xi$ is a  positive constant. 
Then, it is clear that
\begin{equation}
\label{eq:is31inter}
|\fg|_{1,E} \leq |\fg_1|_{1,E} + |\fg_2|_{1,E} \leq \max \{b_1, \,  b_2 \} \, (1 + \xi)  (\|q\|_{0,E} + \|\mathbf{h}\|_{\GGq} ).
\end{equation}
Moreover, by \eqref{eq:dcon} and since $\dd \, \cc = 0$, we have
\begin{equation}
\label{eq:is32inter}
\begin{split}
b^E(\fg, q) +  d^E(\fg, \mathbf{h}) &=  b^E(\fg_1, q) +  d^E(\fg_1, \mathbf{h}) + \xi \, b^E(\fg_2, q) +  \xi \,d^E(\fg_2, \mathbf{h}) \\
& = b^E(\fg_1, q) +   d^E(\fg_1, \mathbf{h})  +  \xi \, d^E(\fg_2, \mathbf{h}) \\
& \geq c_1 \, \|q\|^2_{0,E} +  c_2 \, \xi \, \|\mathbf{h}\|^2_{\GGq} +  d^E(\fg_1, \mathbf{h}) \\
&  \geq c_1 \, \|q\|^2_{0,E} +  c_2 \, \xi \, \|\mathbf{h}\|^2_{\GGq} - c_{\rm cont}  \, |\fg_1|_{1,E} \|\mathbf{h}\|_{\GGq}\\
&  \geq c_1 \, \|q\|^2_{0,E} +  c_2 \, \xi \, \|\mathbf{h}\|^2_{\GGq} - c_{\rm cont} b_1 \, \|q\|_{0,E} \|\mathbf{h}\|_{\GGq}\\
& \geq \left( c_1 - \frac{\epsilon}{2} c_{\rm cont} b_1 \right) \, \|q\|^2_{0,E} + 
\left( \xi \, c_2 - \frac{1}{2 \epsilon} c_{\rm cont} b_1  \right) \, \|\mathbf{h}\|^2_{\GGq}
\end{split}
\end{equation} 
for any positive real number $\epsilon$. Finally, setting
\[
\epsilon := \frac{c_1}{ c_{\rm cont} b_1} \qquad \text{and} \qquad \xi := \frac{ c_{\rm cont}^2 \, b_1^2}{c_1 c_2}
\]
by \eqref{eq:is31inter} and \eqref{eq:is32inter} we get \eqref{eq:is3inter}. 

\noindent
\textit{Step 3.}
Since problem \eqref{eq:pbinter} is well-posed, the following stability estimate holds
\[
|\te|_{1, E} + \|\widehat{s}\|_{0,E} +  \|\widehat{\mathbf{g}}\|_{\GGq} \leq  \|\xx\|_{\left(\GGq \right)^{*}},
\]
where
\[
\|\xx\|_{\left(\GGq \right)^{*}} := \sup_{\mathbf{h} \in \GGq, \mathbf{h} \neq \mathbf{0}} \frac{(\xx, \mathbf{h})}{\|\mathbf{h}\|_{\GGq}} \leq h_E^{-1} \, \|\xx\|_{0,E}.
\]
Then, by the definition of $\xx$ (see \eqref{eq:chi}) and by the continuity of the $L^2$-projection, we get
\[
\begin{split}
|\te|_{1, E} \leq h_E^{-1} \, \|\xx\|_{0,E} \leq h_E^{-1} \, \left \| \PN \ww_I - \ww_I \right\|_{0,E} \leq C \left | \PN \ww_I - \ww_I \right|_{1,E}
\end{split}
\]
where the last inequality is justified since, by definition \eqref{eq:Pn_k^E}, the function $\left ( \PN \ww_I - \ww_I \right)$ has zero mean value. Noting that $\PN$ is a projection with respect the $H^1$ semi-norm and using a triangular inequality, from \eqref{eq:interpolata old} we finally get
\begin{equation}
\label{eq:inter5}
\begin{split}
|\te|_{1,E} & \leq \, \left(  \left | \PN (\ww_I - \vv) |_{1,E} \right| + \left| \ww_I  - \PN \vv \right|_{1,E} \right) \\
& \leq \left(  2 \left | (\ww_I - \vv) |_{1,E} \right| + \left| \vv  - \PN \vv \right|_{1,E} \right) \\
& \leq C \,  h_E^s \, |\vv|_{s+1, E}.
\end{split}
\end{equation}
The thesis now follows from \eqref{eq:inter5} and again \eqref{eq:interpolata old}, by adding all the local contributions.
For what concerns the $L^2$ estimate,  for each polygon $E \in \Omega_h$, we have that $\te = \mathbf{0}$ on $\partial E$ (see \eqref{eq:inter1}). Hence, from \eqref{eq:inter5} it
holds
\[
\|\te\|_{0,E} \leq C \, h_E \, |\te|_{1,E} \leq C \,  h_E^{s+1} \, |\vv|_{s+1, E},
\]
from which we easily infer the $L^2$ estimate.
\end{proof} 

\subsection{Convergence analysis}
\label{sub:4.2}
First of all, let us recall a classical approximation result for $\Pk_k$ polynomials on star-shaped domains, see for instance \cite{brennerscott}.

\begin{lemma}
\label{lm:scott}
Let $E \in \Omega_h$, and let two real numbers $s,p$ with $0 \le s \le k$ and $1 \le p \le \infty$.
Then for all $\mathbf{u} \in [H^{s+1}(E)]^2$, there exists a  polynomial function $\mathbf{u}_{\pi} \in [{\Pk}_k(E)]^2$,  such that
\begin{equation}
\label{eq:scott}
\|\mathbf{u} - \mathbf{u}_{\pi}\|_{L^p(E)} + h_E 
| \mathbf{u} - \mathbf{u}_{\pi} |_{W^{1,p}(E)} \leq C h_E^{s+1}| \mathbf{u}|_{W^{s+1,p}(E)},
\end{equation}
with $C$ depending only on $k$ and the shape regularity constant $\rho$ in assumption $\mathbf{(A1)}$.
\end{lemma}
Now we prove two technical lemmata.
\begin{lemma}
\label{lemma3}
Let $\vv \in H^{s+1}(\Omega) \cap \VV$ with $0 \le s \le k$. Then for all $\ww \in \VV$ it holds
\[
\left|\widetilde{c}(\vv; \, \vv, \ww) - \widetilde{c}_h(\vv; \, \vv, \ww)\right| \leq C \, h^s \, \left( \|\vv\|_s + \|\vv\|_{\VV} + \|\vv\|_{s+1} \right)\|\vv\|_{s+1} \, \|\ww\|_{\VV}.
\]
\end{lemma}

\begin{proof}
First of all, we set
\begin{equation}
\label{eq:mu}
\mu_1(\ww) := \sum_{E \in \Omega_h} \left( c^E(\vv; \, \vv, \ww) - c^E_h(\vv; \, \vv, \ww) \right)
\quad \text{and} \quad
\mu_2(\ww) := \sum_{E \in \Omega_h} \left( c^E(\vv; \, \ww, \vv) - c^E_h(\vv; \, \ww, \vv) \right)
\end{equation}
then by definition \eqref{eq:ctilde_h^E} and \eqref{eq:ctilde_h} it holds
\begin{equation}
\label{eq:mu1mu2}
\widetilde{c}(\vv; \, \vv, \ww) - \widetilde{c}_h(\vv; \, \vv, \ww) = 
 \frac{1}{2} \bigl( \mu_1(\ww) + \mu_2(\ww)\bigr).
\end{equation}
We now analyse the two terms. For the term $\mu(\ww)$ by simple computations, we have
\begin{equation*}
\begin{split}
\mu_1(\ww) 
&  =  \sum_{E \in \Omega_h}   \int_E \left( (\Gr \vv) \, \vv \cdot \ww  -  \left(\PP0 \, \Gr \vv \right)  \left(\P0 \vv \right) \cdot \P0 \ww \right) \, {\rm d}E \\
&  =  \sum_{E \in \Omega_h} \sum_{i,j=1}^2    \int_E \left( \frac{\partial \vv_i}{\partial x_j} \, \vv_j \, \ww_i - \left( \Pi^{0,E}_{k-1} \, \frac{\partial \vv_i}{\partial x_j} \right) \left( \P0 \vv_j\right) \P0 \ww_i \right) \, {\rm d}E\\
\end{split}
\end{equation*}
from which it follows
\begin{equation}
\begin{split}
\label{eq:mu1}
\mu_1(\ww) &=  \sum_{E \in \Omega_h}
 \sum_{i,j=1}^2   \int_E \left( \frac{\partial \vv_i}{\partial x_j} \, \vv_j \left[ \left(I - \P0 \right) \ww_i\right] + 
 \right. \\
& \qquad \qquad  \left. + \frac{\partial \vv_i}{\partial x_j} \left[ \left(I - \P0 \right) \vv_j\right] \, \P0 \, \ww_i +
\left[ \left(I - \Pi^{0,E}_{k-1} \right) \frac{\partial \vv_i}{\partial x_j} \right] \left( \P0 \vv_j\right) \P0 \ww_i \right) \, {\rm d}E \\
& =:  \sum_{E \in \Omega_h}
 \sum_{i,j=1}^2  \int_E \left( \alpha(\ww) + \beta(\ww) + \gamma(\ww) \right) {\rm d}E. 
\end{split}
\end{equation}
Now, by definition of $L^2$ projection $\P0$ and by Lemma \ref{lm:scott}, we have
\begin{equation}
\label{eq:mu1alpha}
\begin{split}
\int_E \alpha(\ww) \,{\rm d}E & =    \int_E  \frac{\partial \vv_i}{\partial x_j} \, \vv_j \left[ \left(I - \P0 \right) \ww_i\right] {\rm d}E \\
& =   \int_E \left[ \left(I - \Pi^{0,E}_{k-2} \right)\frac{\partial \vv_i}{\partial x_j} \, \vv_j \right] \left[ \left(I - \P0 \right) \ww_i\right] {\rm d}E \\
& \leq \left \|\left(I - \Pi^{0,E}_{k-2} \right)\frac{\partial \vv_i}{\partial x_j} \, \vv_j \right \|_{0,E} \, \left\| \left(I - \P0 \right) \ww_i\right\|_{0,E} \\
& \leq C  \, h_E^s \, \left| \frac{\partial \vv_i}{\partial x_j} \, \vv_j\right|_{s-1,E} |\ww_i|_{1,E}.
\end{split}
\end{equation}
Applying H\"older inequality (for sequences), we get
\begin{equation}
\label{eq:mu1alpha1}
\begin{split}
\sum_{E \in \Omega_h} \sum_{i,j=1}^2 \int_E \alpha(\ww) \,{\rm d}E & \leq  C \, h^s \, \sum_{E \in \Omega_h} \sum_{i,j=1}^2 \,\left| \frac{\partial \vv_i}{\partial x_j} \, \vv_j\right|_{s-1,E} |\ww_i|_{1,E} \\
&\leq  C \, h^s \, \sum_{i,j=1}^2 \left( \sum_{E \in \Omega_h}  \,\left| \frac{\partial \vv_i}{\partial x_j} \, \vv_j\right|^2_{s-1,E} \right)^{\frac{1}{2}} \left( \sum_{E \in \Omega_h}  |\ww_i|^2_{1,E} \right)^{\frac{1}{2}}\\
&\leq  C \, h^s \, \sum_{i,j=1}^2 \, \left| \frac{\partial \vv_i}{\partial x_j} \, \vv_j\right|_{s-1} \, |\ww_i|_{1}
\end{split}
\end{equation}
and by H\"older inequality and Sobolev embedding $H^{s-1}(\Omega) \subset W^{s}_4(\Omega)$ we infer
\begin{equation}
\label{eq:mu1alpha2}
\left| \frac{\partial \vv_i}{\partial x_j} \, \vv_j\right|_{s-1} \leq 
\left \| \frac{\partial \vv_i}{\partial x_j}  \right\|_{W^{s-1}_4} \,
\left \| \vv_j \right\|_{W^{s-1}_4} 
\leq C \,
\left \| \frac{\partial \vv_i}{\partial x_j}  \right\|_{s} \,
\left \| \vv_j \right\|_{s} .
\end{equation}
By \eqref{eq:mu1alpha1} and \eqref{eq:mu1alpha2} we finally  obtain
\begin{equation}
\label{eq:alphafinale}
\sum_{E \in \Omega_h} \sum_{i,j=1}^2 \alpha(\ww)  \leq  C \, h^s \left \| \vv  \right\|_{s+1} \,
\left \| \vv \right\|_{s} \, \|\ww\|_{\VV}.
\end{equation}
For what concerns the  term $\beta(\ww)$ in \eqref{eq:mu1} using H\"older inequality we have
\begin{equation}
\label{eq:mu1beta}
\begin{split}
\int_E \beta(\ww) \,{\rm d}E & =    \int_E \frac{\partial \vv_i}{\partial x_j} \left[ \left(I - \P0 \right) \vv_j\right] \, \P0 \, \ww_i \, {\rm d}E \\
& \leq 
  \left \| \frac{\partial \vv_i}{\partial x_j} \right\|_{0,E} \,
\left \| \left(I - \P0 \right) \vv_j \right\|_{L^4(E)} \, 
\left\| \P0 \, \ww_i \right\|_{L^4(E)}.
\end{split}
\end{equation}
Lemma \ref{lm:scott} yields a polynomial $\vv_{j, \pi} \in \Pk_k(E)$ such that
\[
\|\vv_j - \vv_{j, \pi} \|_{L^4(E)} \leq C \, h_E^s \, |\vv_j|_{W^s_4(E)}
\]
and thus, by the continuity of $\P0$ with respect the $L^4$-norm (cf. \eqref{eq:continuity-ch3}),
\begin{equation}
\label{eq:mu1beta2}
\begin{aligned}
\left \| \left(I - \P0 \right) \vv_j \right\|_{L^4(E)}  & \leq \|\vv_j - \vv_{j, \pi} \|_{L^4(E)}  + \left\|\P0 \, (\vv_j - \vv_{j, \pi}) \right\|_{L^4(E)} \\
& \leq  C \|\vv_j - \vv_{j, \pi} \|_{L^4(E)} \leq  C \, h_E^s \, |\vv_j|_{W^s_4(E)}.
\end{aligned}
\end{equation}
Using again the continuity of $\P0$ with respect the $L^4$-norm, by \eqref{eq:mu1beta} and \eqref{eq:mu1beta2} we infer
\[
\int_E \beta(\ww) \,{\rm d}E  \leq 
C\, h_E^s \, \left \| \frac{\partial \vv_i}{\partial x_j} \right\|_{0,E} \,|\vv_j|_{W^s_4(E)} \, \| \ww_i \|_{L^4(E)}.
\]
Applying the H\"older inequality and Sobolev embeddings $H^{1}(\Omega) \subset L^4(\Omega)$ and $H^{s+1}(\Omega) \subset W^{s}_4(\Omega)$, we obtain
\begin{equation}
\label{eq:betafinale}
\begin{split}
\sum_{E \in \Omega_h} & \sum_{i,j=1}^2 \int_E \beta(\ww) \,{\rm d}E  \leq  C \, h^s \, \sum_{E \in \Omega_h} \sum_{i,j=1}^2 \, \left \| \frac{\partial \vv_i}{\partial x_j} \right\|_{0,E} \,|\vv_j|_{W^s_4(E)} \, \| \ww_i \|_{L^4(E)}\\
&\leq  C \, h^s \, \sum_{i,j=1}^2 
\left( \sum_{E \in \Omega_h}  \,  \left \| \frac{\partial \vv_i}{\partial x_j} \right\|^2_{0,E} \right)^{\frac{1}{2}} 
\left( \sum_{E \in \Omega_h}  |\vv_j|^4_{W^s_4(E)} \right)^{\frac{1}{4}}
\left( \sum_{E \in \Omega_h} \ |\ww_i\|^4_{L^4(E)} \right)^{\frac{1}{4}}
\\
&\leq  C \, h^s \, \sum_{i,j=1}^2 \, \left\| \frac{\partial \vv_i}{\partial x_j} \right \|_0\, \|\vv_j \|_{W^s_4} \, \|\ww_i\|_{L^4} \leq  C \, h^s \, \left\| \vv \right \|_{\VV} \, \|\vv\|_{s+1} \, \|\ww\|_{\VV}.
\end{split}
\end{equation}
For what concerns the term $\gamma(\ww)$ in \eqref{eq:mu1}, using H\"older and the continuity of $\P0$, it holds
\begin{equation}
\label{eq:mu1gamma}
\begin{split}
\int_E \gamma(\ww)  \,{\rm d}E  &=    \int_E \left[ \left(I - \Pi^{0,E}_{k-1} \right) \frac{\partial \vv_i}{\partial x_j} \right] \left( \P0 \vv_j\right) \P0 \ww_i  \, {\rm d}E  \\
& \leq 
  \left \| \left(I - \Pi^{0,E}_{k-1} \right) \frac{\partial \vv_i}{\partial x_j} \right\|_{0,E} \,
 \|  \P0 \, \vv_j\|_{L^4(E)} \, 
\| \P0 \, \ww_i \|_{L^4(E)} \\
& \leq 
  C \, h_E^s \, \left |  \frac{\partial \vv_i}{\partial x_j} \right|_{s,E} \, \|  \vv_j\|_{L^4(E)} \, 
\| \ww_i \|_{L^4(E)}  . \\
\end{split}
\end{equation}
Using again the H\"older inequality and Sobolev embedding we get
\begin{equation}
\label{eq:gammafinale}
\sum_{E \in \Omega_h} \sum_{i,j=1}^2 \, \int_E \gamma(\ww) \,{\rm d}E  \leq  C \, h^s \, \|\vv\|_{\VV} \, \|\ww\|_{\VV} \, \|\vv\|_{s+1}.
\end{equation}
By collecting \eqref{eq:alphafinale}, \eqref{eq:betafinale} and \eqref{eq:gammafinale} in \eqref{eq:mu1} we finally get
\begin{equation}
\label{eq:mu1finale}
 \mu_1(\ww) \leq C \, h^s \, \left( \|\vv\|_{s+1} \|\vv\|_{s} +\|\vv\|_{s+1} \|\vv\|_{\VV} \right) \|\ww\|_{\VV}.
\end{equation} 
For the second term $\mu_2(\ww)$ we only sketch the proof since we use analogous arguments. First by definition, then by adding and subtracting terms, we obtain 
\begin{equation}
\begin{split}
\label{eq:mu2}
\mu_2(\ww) &= \sum_{E \in \Omega_h}  \sum_{i,j=1}^2   \int_E \left (
\left[ \left(I - \Pi^{0,E}_{k-1} \right) \frac{\partial \ww_i}{\partial x_j} \right]  \vv_j\, \vv_i + \left( \Pi^{0,E}_{k-1} \,\frac{\partial \ww_i}{\partial x_j} \right) \left[ \left( I - \P0 \right) \vv_j \right] \vv_i + \right. \\
& \qquad \qquad \left. 
 + \left( \Pi^{0,E}_{k-1} \, \frac{\partial \ww_i}{\partial x_j} \right)  \left( \P0 \, \vv_j\right) \left[ \left( I - \P0 \right) \vv_i \right] \right) \, {\rm d}E \\
 & =: \sum_{E \in \Omega_h}  \sum_{i,j=1}^2 \int_E \left( \delta(\ww) + \epsilon(\ww) + \zeta(\ww)\right) \, {\rm d}E.
\end{split}
\end{equation}
For the term $\delta(\ww)$ we have
\begin{equation}
\label{eq:mu2delta}
\begin{split}
\int_E \delta(\ww) \,{\rm d}E & =    \int_E  \left[ \left(I - \Pi^{0,E}_{k-1} \right) \frac{\partial \ww_i}{\partial x_j} \right]  \vv_j\, \vv_i \, {\rm d}E \\
& =    \int_E  \left[ \left(I - \Pi^{0,E}_{k-1} \right) \frac{\partial \ww_i}{\partial x_j} \right] \left[ \left(I - \Pi^{0,E}_{k-1}\right) \vv_j\, \vv_i \right] \, {\rm d}E \\
& \leq \left \|\left(I - \Pi^{0,E}_{k-1} \right)\frac{\partial \ww_i}{\partial x_j} \right \|_{0,E} \, \left\| \left(I - \Pi^{0,E}_{k-1} \right) \vv_j\, \vv_i \right\|_{0,E} \\
& \leq C \sum_{i,j=1}^2 \, h_E^s \, \left\| \frac{\partial \ww_i}{\partial x_j} \right\|_{0,E} \, 
|\vv_j \, \vv_i |_{s,E}
\end{split}
\end{equation}
and applying the H\"older inequality (for sequences) we easily get
\[
\sum_{E \in \Omega_h} \sum_{i,j=1}^2 \int_E \delta(\ww) \,{\rm d}E \leq  
C \, h^s \, \sum_{i,j=1}^2 \, \left\| \frac{\partial \ww_i}{\partial x_j} \right\|_{0} \, |\vv_j \, \vv_i |_{s} .
\]
The H\"older inequality and the Sobolev embedding $H^{s+1}(\Omega) \subset W^{s}_4(\Omega)$ yield
\[
|\vv_j \, \vv_i |_{s} \leq \|\vv_j\|_{W^s_4} \,\|\vv_i\|_{W^s_4} \leq C \, \|\vv_j\|_{s+1} \,\|\vv_i\|_{s+1}
\]
and thus we conclude that
\begin{equation}
\label{eq:deltafinale}
\sum_{E \in \Omega_h} \sum_{i,j=1}^2  \delta(\ww)  \leq  C \, h^s \, \|\vv\|_{s+1}^2 \, \|\ww\|_{\VV}.
\end{equation}
The terms $\epsilon(\ww)$ and $\zeta(\ww)$ can be estimated using the usual argument (H\"older inequality, continuity of $\P0$ with respect to the $L^4$-norm and Sobolev embeddings). We conclude that
\begin{equation}
\label{eq:mu2finale}
\mu_2(\ww) \leq C \, h^s \, \left( \|\vv\|_{s+1}^2  +\|\vv\|_{s+1} \|\vv\|_{\VV} \right) \|\ww\|_{\VV}.
\end{equation} 
We infer the thesis by collecting \eqref{eq:mu1finale} and \eqref{eq:mu2finale} in \eqref{eq:mu1mu2}.
\end{proof}

\begin{lemma}
\label{lemma2}
Let $\widehat{C}_h$ be the constant defined in \eqref{eq:CC}. Then for all $\vv, \mathbf{z}, \ww \in \VV$ it holds
\[
|\widetilde{c}_h(\vv; \, \vv, \ww) - \widetilde{c}_h(\mathbf{z}; \, \mathbf{z}, \ww)| \leq   
\widehat{C}_h \,  \left( \|\mathbf{z}\|_{\VV} \, \| \ww\|_{\VV} + \|\vv - \mathbf{z} + \ww\|_{\VV} (\|\vv\|_{\VV} + \|\mathbf{z}\|_{\VV}) \right) \, \|\ww\|_{\VV}.
\]
\end{lemma}

\begin{proof}
Since $\widetilde{c}_h(\cdot; \, \cdot, \cdot)$ is skew-symmetric by simple computations we obtain
\[
\begin{split}
\widetilde{c}_h(\vv; \, \vv, \ww) - \widetilde{c}_h(\mathbf{z}; \, \mathbf{z}, \ww) & =
\widetilde{c}_h(\vv - \mathbf{z}; \, \vv, \ww) + \widetilde{c}_h(\mathbf{z}; \, \vv - \mathbf{z}, \ww) \\
& =  
-\widetilde{c}_h(\ww; \,  \vv, \ww) + \widetilde{c}_h(\vv - \mathbf{z} + \ww; \, \vv, \ww) + \widetilde{c}_h(\mathbf{z}; \, \vv - \mathbf{z} + \ww, \ww).
\end{split}
\]
The thesis follows by definition \eqref{eq:CC}.
\end{proof}

Furthermore, we state the following result concerning the load approximation, which can be proved using standard arguments \cite{volley}.
\begin{lemma}
\label{lemma4}
Let $\mathbf{f}_h$ be defined as in \eqref{eq:f_h},  and let us assume $\mathbf{f} \in H^{s+1}(\Omega)$, $-1 \le s \le k$. Then, for all $\mathbf{v}_h \in \mathbf{V}_h$, it holds
\begin{gather*}
\left|( \mathbf{f}_h - \mathbf{f}, \mathbf{v}_h ) \right| \leq C h^{s+2} |\mathbf{f}|_{s+1} |\mathbf{v}_h|_{\mathbf{V}}.
\end{gather*}
\end{lemma}
We now note that, given $\vv \in \ZZ$, the inf-sup condition \eqref{eq:inf-sup discreta} implies (see \cite{BoffiBrezziFortin}):
\[
\inf_{\vv_h \in \ZZ_h, \vv_h \neq \mathbf{0}} \|\vv - \vv_h\|_{\VV} \leq C
\inf_{\ww_h \in \VV_h, \ww_h \neq \mathbf{0}} \|\vv - \ww_h\|_{\VV}
\]
which essentially means that $\ZZ$ is approximated by $\ZZ_h$ with the same accuracy order of the
whole subspace $\VV_h$. In particular by Theorem \ref{thm:interpolante}, assuming 
$\vv \in H^{s+1}(\Omega) \cap \ZZ$, $0 < s \le k$, we infer
\begin{equation}
\label{eq:interpolant kernel}
\inf_{\vv_h \in \ZZ_h, \vv_h \neq \mathbf{0}} \|\vv - \vv_h\|_{\VV} \leq C \, h^s \, | \vv |_{s+1}.
\end{equation}

\begin{theorem}
\label{thm:u}
Under the assumptions \eqref{eq:ns condition} and \eqref{eq:ns virtual condition}, let $\uu$ be the solution of Problem \eqref{eq:ns variazionale ker} and $\uu_h$ be the solution of virtual Problem \eqref{eq:nsvirtual ker}. Assuming moreover $\uu, \ff \in [H^{s+1}(\Omega)]^2$, $0 < s \le k$, then 
\begin{equation}\label{eq:thm:u}
\| \uu - \uu_h \|_{\VV} \leq \, h^{s} \, \mathcal{F}(\uu; \, \nu, \gamma, \gamma_h) + \, h^{s+2} \, \mathcal{H}(\ff; \nu, \gamma_h)
\end{equation}
where $\mathcal{F}$ and $\mathcal{H}$ are suitable functions independent of $h$.
\end{theorem}

\begin{proof}
Let $\uu_I$ be an approximant of $\uu$ in the discrete kernel $\ZZ_h$ satisfying \eqref{eq:interpolant kernel}, and let us define $\ddd_h := \uu_h - \uu_I$. Now, by the stability and the consistency properties (cf. \eqref{eq:consist} and \eqref{eq:stabk}) of the bilinear form $a_h(\cdot, \cdot)$, the triangular inequality and \eqref{eq:interpolant kernel} give
\begin{equation}
\label{eq:thm1}
\begin{split}
\alpha_* \, \nu \, \| \ddd_h \|^2_{\VV}  & \leq \nu \, a_h(\ddd_h, \, \ddd_h) =  \nu \, a_h(\uu_h, \, \ddd_h) -  \nu \, a_h(\uu_I, \, \ddd_h)  
\\
& = \nu \, a_h(\uu_h, \, \ddd_h) - \nu\,  a(\uu, \, \ddd_h) +  \nu \sum_{E \in \Omega_h} \left( a_h^E(\uu_{\pi} - \uu_I, \, \ddd_h)  + a^E(\uu - \uu_{\pi}, \, \ddd_h) \right)  \\
& \leq \nu \, a_h(\uu_h, \, \ddd_h) - \nu\,  a(\uu, \, \ddd_h) +  
C \, \nu \, h^s |\uu |_{s+1}\| \ddd_h \|_{\VV} \\
\end{split}
\end{equation}
where $\uu_{\pi}$ is the piecewise polynomial of degree $k$ defined  in Lemma \ref{lm:scott}.
Now since $\uu$ and $\uu_h$ are solutions of Problem \eqref{eq:ns variazionale ker} and Problem \eqref{eq:nsvirtual ker}  respectively, from Lemma \ref{lemma4} we obtain 
\begin{equation}
\label{eq:thm2}
\begin{split}
\alpha_* \, \nu \, \| \ddd_h \|^2_{\VV}  
& \leq (\ff_h - \ff,\, \ddd_h) +  \widetilde{c}(\uu; \, \uu, \ddd_h) - \widetilde{c}_h(\uu_h; \, \uu_h, \ddd_h)    +  C \, \nu \, h^s | \uu |_{s+1}\|\ddd_h\|_{\VV} \\
& \leq C \, h^s (\nu \, |\uu |_{s+1} + h^2 \, |\ff|_{s+1})\|\ddd_h\|_{\VV} + \widetilde{c}(\uu; \, \uu, \ddd_h)  - \widetilde{c}_h(\uu_h; \, \uu_h, \ddd_h).
\end{split}
\end{equation}
Now we observe that
\begin{equation}
\label{eq:lemma3thm}
\widetilde{c}(\uu; \, \uu, \ddd_h)  - \widetilde{c}_h(\uu_h; \, \uu_h, \ddd_h) = 
\bigl(\widetilde{c}(\uu; \, \uu, \ddd_h) - \widetilde{c}_h(\uu; \, \uu, \ddd_h) \bigr) +
\bigl(\widetilde{c}_h(\uu; \, \uu, \ddd_h) - \widetilde{c}_h(\uu_h; \, \uu_h, \ddd_h) \bigr).
\end{equation}
The first term can be estimated by Lemma \ref{lemma3}
\[
\widetilde{c}(\uu; \, \uu, \ddd_h) - \widetilde{c}_h(\uu; \, \uu, \ddd_h) \leq C \, h^s \, \left( \|\uu\|_s + \|\uu\|_{\VV} + \|\uu\|_{s+1} \right)\|\uu\|_{s+1} \, \|\ddd_h\|_{\VV}.
\] 
The second term, recalling that $\ddd_h = \uu_h - \uu_I$, is bounded by Lemma \ref{lemma2}
\begin{equation}
\label{eq:lemma2thm}
\widetilde{c}_h(\uu; \, \uu, \ddd_h) - \widetilde{c}_h(\uu_h; \, \uu_h, \ddd_h) \leq
\widehat{C}_h \,  \bigl( \|\uu_h\|_{\VV} \, \| \ddd_h\|_{\VV} + \|\uu - \uu_I\|_{\VV} (\|\uu\|_{\VV} + \| \uu_h\|_{\VV}) \bigr) \, \|\ddd_h\|_{\VV}.
\end{equation}
Collecting \eqref{eq:lemma3thm} and \eqref{eq:lemma2thm} in \eqref{eq:thm2}, we get
\begin{multline}
\label{eq:thm3}
\alpha_* \, \nu \, \| \ddd_h \|_{\VV}   \leq  C \, h^s (\nu \, |\uu|_{s+1} + h^2 \, |\ff|_{s+1}) + 
C \, h^s \, \left( \|\uu\|_s + \|\uu\|_{\VV} + \|\uu\|_{s+1} \right)\|\uu\|_{s+1} + \\
 + \widehat{C}_h \,  \bigl( \|\uu_h\|_{\VV} \, \| \ddd_h\|_{\VV} + \|\uu - \uu_I\|_{\VV} (\|\uu\|_{\VV} + \| \uu_h\|_{\VV}) \bigr)
\end{multline}
and then by  Theorem \ref{thm:interpolante} we infer
\begin{multline}
\label{eq:thm4}
\alpha_* \, \nu \left( 1 - \frac{\widehat{C}_h \, \|\uu_h\|_{\VV} }{\alpha_* \, \nu}\right) \, \| \ddd_h \|_{\VV}   \leq  C \, h^s (\nu \, |\uu|_{s+1} + h^2 \, |\ff|_{s+1}) + \\
 + C \, h^s \, \left( \|\uu\|_s + \|\uu\|_{\VV} + \|\uu\|_{s+1} \right)\|\uu\|_{s+1} +   C \, h^s \, \|\uu\|_{s+1} \, \widehat{C}_h \,(\|\uu\|_{\VV} + \| \uu_h\|_{\VV}).
\end{multline}
We observe now that from \eqref{eq:solution virtual estimates} and \eqref{eq:ns virtual condition}, it holds
\[
1 - \frac{\widehat{C}_h \, \|\uu_h\|_{\VV} }{\alpha_* \, \nu} \geq 1 -  \frac{\widehat{C}_h \, \|\ff_h\|_{H^{-1}} }{(\alpha_* \, \nu)^2} \ge 1 - r > 0.
\]
Therefore
\begin{multline*}
\| \ddd_h \|_{\VV}   \leq  
C \, \frac{h^s}{1 - r}  \left(|\uu|_{s+1} +   \frac{h^{2} }{\nu} \|\ff\|_{s+1}\right) + 
C \, \frac{h^s}{\nu(1 - r)} \, \left( \|\uu\|_s + \|\uu\|_{\VV} + \|\uu\|_{s+1} \right)\|\uu\|_{s+1} \\ +
C \, h^s \, \|\uu\|_{s+1} \, \frac{\widehat{C}_h}{  \nu (1 - \gamma_h)} \,(\|\uu\|_{\VV} + \| \uu_h\|_{\VV})
\end{multline*}
and from \eqref{eq:solution estimates}, \eqref{eq:ns condition},  \eqref{eq:solution virtual estimates} and \eqref{eq:ns virtual condition} we finally obtain
\begin{multline*}
\| \ddd_h \|_{\VV}   \leq  
C \, \frac{h^s}{1 - r}  \left(|\uu|_{s+1} +   \frac{h^{2} }{\nu} \|\ff\|_{s+1}\right) + 
C \, \frac{h^s}{\nu(1 - r)} \, \left( \|\uu\|_s + \|\uu\|_{\VV} + \|\uu\|_{s+1} \right)\|\uu\|_{s+1} + \\ +
C \, h^s \, \|\uu\|_{s+1} \, \left(  \frac{\widehat{C}_h}{ \widehat{C}} \frac{\gamma}{1 - r} + \frac{\gamma_h}{1 - r} \right).
\end{multline*}
The thesis easily follows from the triangular inequality.
\end{proof}
\begin{remark}
We observe that, due to the divergence-free property of the proposed method, the estimate on the velocity errors in Theorem \ref{thm:u} does not depend on the continuous pressure, whereas the velocity errors of classical methods have a pressure contribution. 
A numerical investigation of this aspect, also in relation to the presence of a higher order load approximation term in the right hand side of 
\eqref{eq:thm:u}, will be shown in the next section.
\end{remark}

\begin{remark}
From the discrete inf-sup condition \eqref{eq:inf-sup discreta} the pressure estimate easily follows by standard arguments. Let $(\uu, p) \in \VV \times Q$ be the solution of Problem \eqref{eq:ns variazionale} and $(\uu_h, p_h) \in \VV_h \times Q_h$ be the solution of Problem \eqref{eq:ns virtual}. Then it holds:
\begin{equation}\label{eq:p-est}
\|p  - p_h\|_Q \leq C \, h^{s} \, |p|_{s} + C \, h^{s+2} \, |\ff|_{s+1} + h^s \, \mathcal{K}(\uu; \nu, \gamma, \gamma_h) 
\end{equation}
for a suitable function $\mathcal{K}(\cdot; \,  \cdot, \cdot)$ independent of $h$. 
\end{remark}

\begin{remark}
In Theorem \ref{thm:u} we have assumed $\uu$ and ${\bf f}$ in $H^{s+1}(\Omega)$. However, it is easy to check that the same analysis can be performed if we only require: 
$$
\uu, \: {\bf f} \in H^{s+1}(E) \quad \forall E \in \Omega_h. 
$$
In such a case, the higher order Sobolev norms on $\uu, {\bf f}$ appearing in Theorem \ref{thm:u} (and in the other results of this section) are substituted with the corresponding element-wise broken Sobolev norms.
\end{remark}

\section{Numerical Tests}
\label{sec:5}

In this section we present six sets of numerical experiments to test the practical performance of the method. 
All the tests are performed with the second-order VEM, i.e. $k=2$. We also consider suitable second order Finite Elements for comparison.
In almost all cases, both options ${c}_h(\cdot;\cdot,\cdot)$ and ${\widetilde c}_h(\cdot;\cdot,\cdot)$ (see Remark \ref{rem:non-skew}) yield very similar results; in such cases, only the first choice is reported. On the contrary, whenever the results between the two choices are significantly different, both outcomes are shown.

In Test \ref{test1} and Test \ref{test2}, we consider two benchmark problems for the Stokes and Navier-Stokes equation. They share the property of having the velocity solution in the discrete space. However, classical mixed finite element methods lead to significant velocity errors, stemming from the velocity/pressure coupling in the error estimates. This effect is greatly reduced (or even neglected) by our VEM methods (cf. Theorem \ref{thm:u} and estimate \eqref{eq:p-est}). 
In Test \ref{test3} we analyse the stability of the method with respect to the viscosity parameter $\nu$. 
In Test \ref{test4} and Test \ref{test5} we study the convergence of the proposed method for the Navier-Stokes and Stokes equations, respectively.
A comparison with the triangular P2-P1 and the quadrilateral Q2-P1 mixed finite element methods, see for example \cite{BoffiBrezziFortin}, is also performed. 
Finally in Test \ref{test6} we assess the proposed virtual element method for flows which are governed by the Stokes system on one part of the domain, and by the Darcy's law in the rest of the domain, the solutions in the two domains being coupled by proper interface conditions (see Remark \ref{rem:Sto-Dar}). \\
In order to compute the VEM errors, we consider the computable error quantities:
\begin{gather*}
{\rm error}(\mathbf{u}, H^1) := \left( \sum_{E \in \Omega_h} \left \| \boldsymbol{\nabla} \, \uu -  \boldsymbol{\Pi}_{k-1}^{0, E} (\boldsymbol{\nabla} \, \uu_h) \right \|_{0,K}^2 \right)^{1/2} 
\\
{\rm error}(\mathbf{u}, L^2)  := \left(\sum_{E \in \Omega_h} \left \| \uu -  \Pi_{k}^{0, E}  \,\uu_h \right \|_{0,E}^2 \right)^{1/2} 
\\
{\rm error}(\mathbf{u}, L^{\infty})  := \max_{\mathbf{x} \in \, {\rm nodes}} | \uu(\mathbf{x}) - \uu_h(\mathbf{x})|
\end{gather*}
where in the previous formula ``nodes'' denotes the set of internal edges nodes and internal vertexes (cf. $\mathbf{D_V1}$ and $\mathbf{D_V2}$). For the pressures we simply compute 
\[
{\rm error}(p, L^2)  := \|p - p_h\|_0.
\]
The polynomial degree of accuracy for the numerical tests is $k=2$.
In the experiments we consider the computational domains $\Omega_{\rm Q} := [0,1]^2$ and $\Omega_{\rm D} := \{\mathbf{x} \in \R^2 \, \text{s.t.} \, |\mathbf{x}| \leq 1 \}$. The square domain $\Omega_{\rm Q}$ is partitioned using the following sequences of polygonal meshes:
\begin{itemize}
\item $\{ \mathcal{Q}_{h} \}_h, \, \{ \mathcal{U}_{h}\}_h$: sequences of distorted quadrilateral meshes with $h=1/10, 1/20, 1/40, 1/80$,
\item $\{ \mathcal{T}_h\}_h$: sequence of triangular meshes with $h=1/5, 1/10, 1/20, 1/40$,
\item $\{ \mathcal{W}_h\}_h$: sequence of WEB-like meshes with $h= 1/5, 1/10,  1/20, 1/40$.
\end{itemize}
An example of the adopted meshes is shown in Figure \ref{meshq}.  
\begin{figure}[!h]
\centering
\includegraphics[scale=0.25]{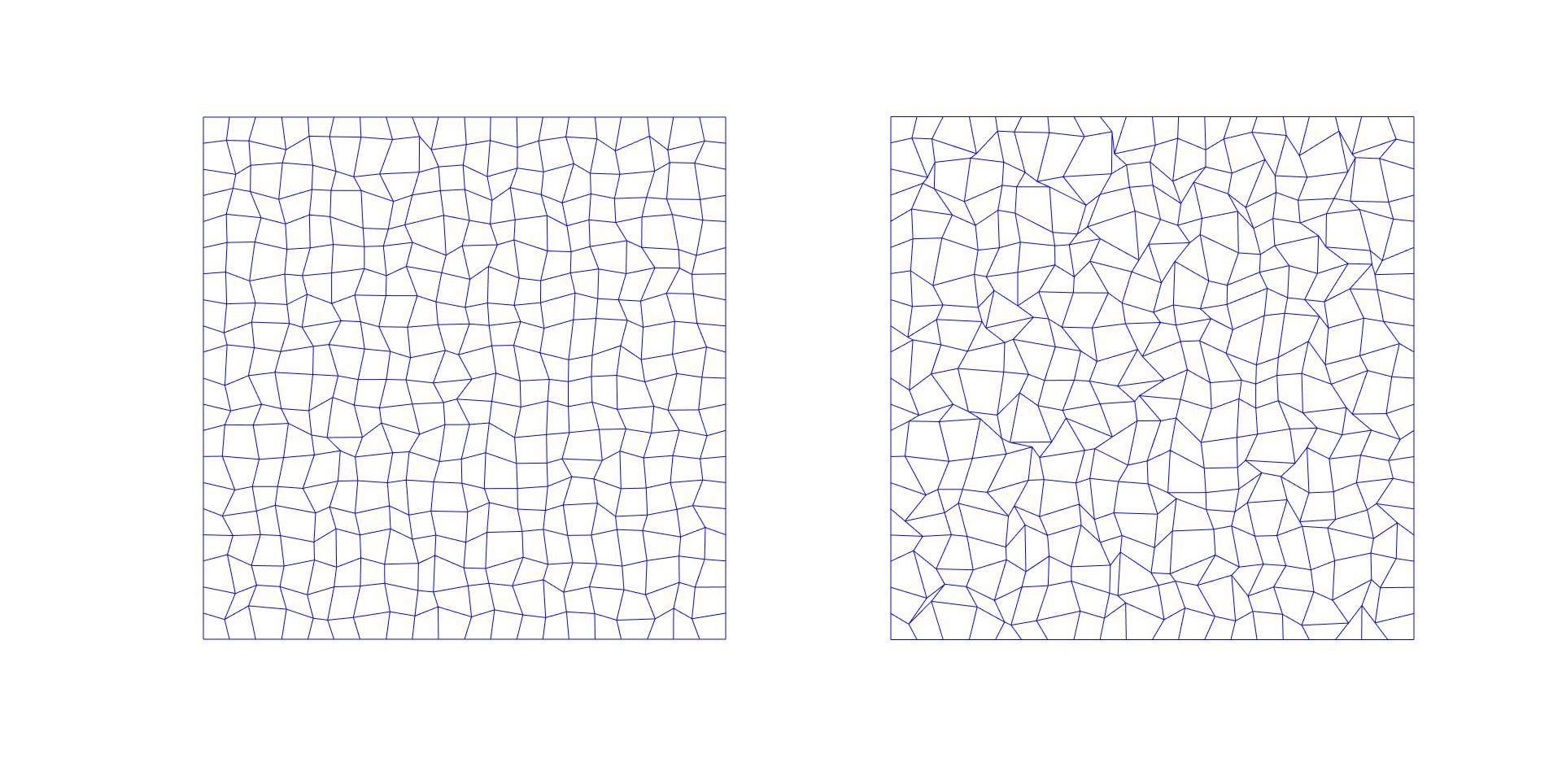} 
\includegraphics[scale=0.25]{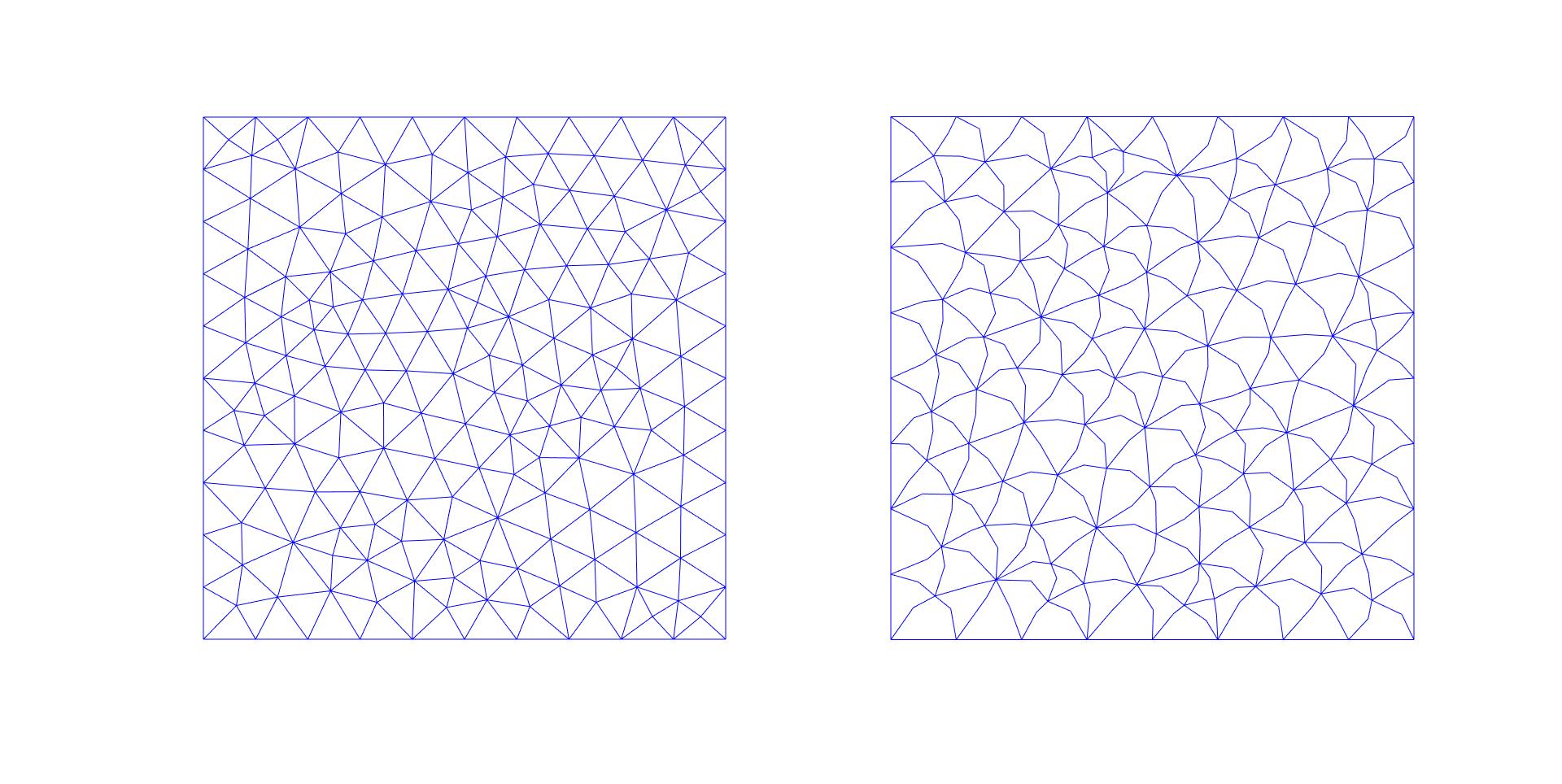} 
\caption{Example of the adopted polygonal meshes: $\mathcal{Q}_{1/20}$, $\mathcal{U}_{1/20}$ (up);  $\mathcal{T}_{1/10}$, $\mathcal{W}_{1/10}$ (down).}
\label{meshq}
\end{figure}
The distorted quadrilateral meshes are obtained starting from the uniform square meshes and displacing the internal vertexes (with a proportional ``distortion amplitude'' of $0.3$ for $\mathcal{Q}_h$ and $0.5$ for $\mathcal{U}_h$). The non-convex WEB-like meshes are composed by hexagons, generated starting from the triangular meshes $\{ \mathcal{T}_h\}_h$ and randomly displacing the midpoint of each (non boundary) edge.
For what concerns the disk $\Omega_{\rm D}$ we consider the sequences of polygonal meshes:
\begin{itemize}
\item $\{ \mathcal{T}_h\}_h$: sequence of triangular meshes with $h=1/5, 1/10, 1/20, 1/40$,
\item $\{ \mathcal{V}_h\}_h$: sequence of CVT Voronoi meshes with $h= 1/5, 1/10, 1/20, 1/40$.
\end{itemize}
Figure \ref{meshd} displays an example of the adopted meshes.
\begin{figure}[!h]
\centering
\includegraphics[scale=0.25]{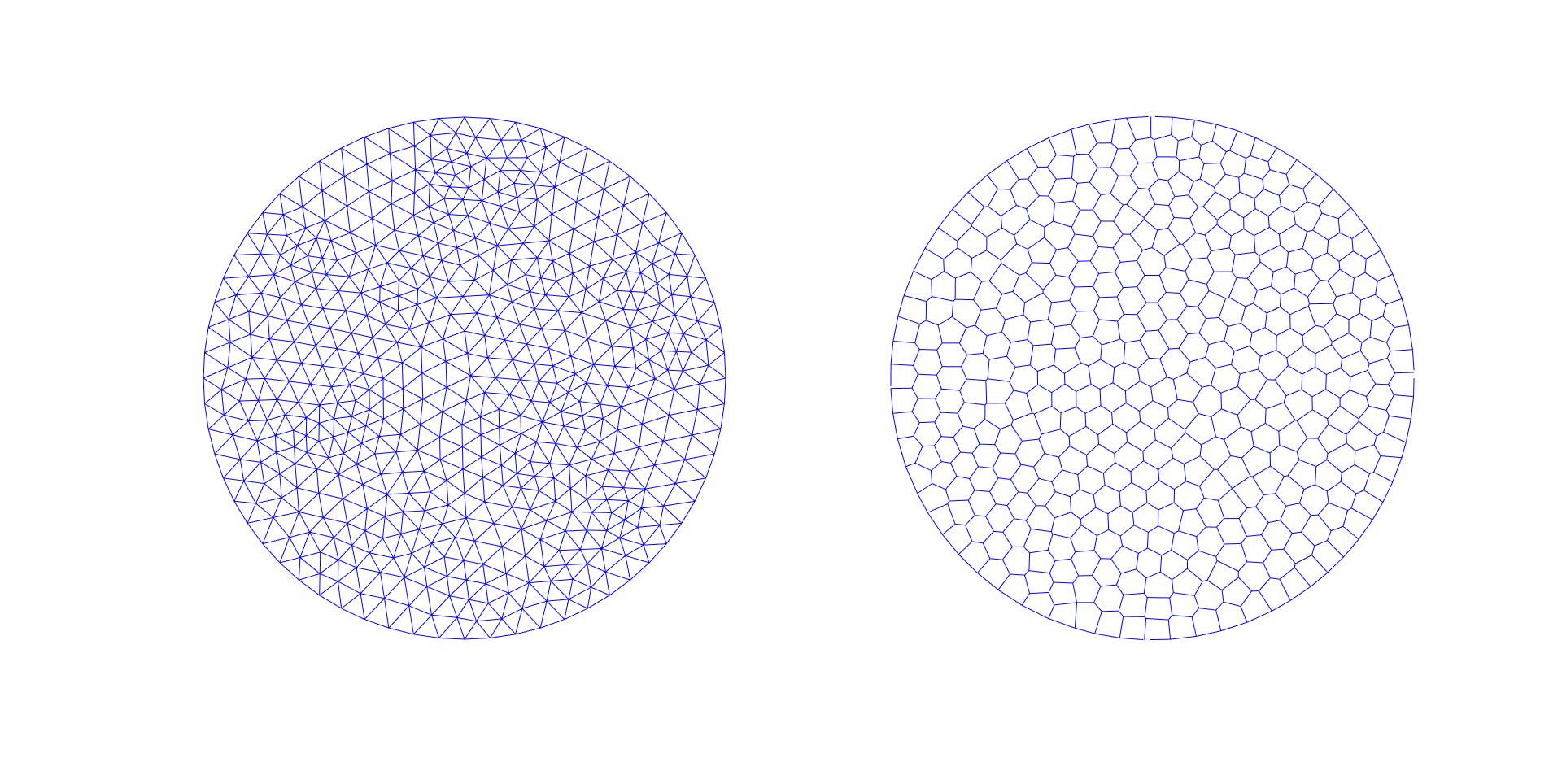} 
\caption{Example of polygonal meshes: $\mathcal{T}_{1/10}$,  $\mathcal{V}_{1/10}$.}
\label{meshd}
\end{figure}
For the generation of the Voronoi meshes we use the code Polymesher \cite{TPPM12}. 

\begin{remark}
As a comparison, we make use also of the classical Q2-P1 and P2-P1 mixed finite elements, see for instance \cite{BoffiBrezziFortin}. The Q2-P1 (Crousiex-Raviart) is a quadrilateral element with bi-quadratic velocities and $\Pk_1$ discontinuous pressures. The P2-P1 (Taylor-Hood) is a triangular element  with $\Pk_2$ velocities and $\Pk_1$ continuous pressures. Both are inf-sup stable elements, widely used in the literature and yielding a quadratic convergence rate in the natural norms of the problem.
\end{remark}


\begin{test} [\textbf{Hydrostatic fluids}]
\label{test1}
In this test we consider the linear Stokes equation on the domain $\Omega_{\rm Q}$ with external load $\ff = \nabla \, p$ that exactly balances the gradient of the pressure, yielding a hydrostatic situation, i.e. $\uu = \mathbf{0}$. 
We set the viscosity $\nu = 1$ and we consider two possible pressures
\[
p_1(x, y) = x^3 - y^3 \qquad \text{and} \qquad p_2(x, y) = \sin(2 \pi x)\sin(2 \pi y).
\]
It is well known that the velocity error between the exact velocity $\uu$ and the discrete velocity $\uu_h$ of standard mixed elements like the Q2-P1 element for the incompressible Stokes equations is pressure-dependent, i.e. has the form
\begin{equation}
\label{eq:errorfem}
\|\uu - \uu_h\|_{\VV} \leq C_1 \, \inf_{\vv_h \VV_h} \|\uu - \vv_h\|_{\VV} + C_2 \, \inf_{q_h \in Q_h} \|p - q_h\|_{Q}
\end{equation}
where $C_1, C_2$ are two positive uniform constants, whereas for the virtual element scheme (see Theorem \ref{thm:u}  and \cite{Stokes:divfree}) the error on the velocity does not depend by the pressure, i.e.
\begin{equation}
\label{eq:errorvem}
\|\uu - \uu_h\|_{\VV} \leq C_1 \, \inf_{\vv_h \VV_h} \|\uu - \vv_h\|_{\VV} + C_2 \, h^{k+2} |\ff|_{k+1}.
\end{equation}
We observe that for both VEM and Q2-P1, the pressures $p_1$ and $p_2$ do not belong to the discrete pressure  space. Therefore we expect that the discrete Q2-P1 velocities are polluted by the pressure approximation. 
Table \ref{tab1} shows the results obtained respectively  with VEM and Q2-P1 for the case of polynomial pressure $p_1$ and sequence of meshes $\mathcal{Q}_h$. 
We observe that the virtual element method yields an exact hydrostatic velocity solution, since $\ff$ is a polynomial of degree two, while the Q2-P1 finite element method, in accordance with the a priori estimate \eqref{eq:errorfem}, shows non-negligible errors in the velocity. 

\begin{table}[!h]
\centering
\begin{tabular}{ll*{3}{c}}
\toprule
&                         $h$                      & ${\rm error}(\mathbf{u}, H^1)$   & ${\rm error}(\mathbf{u}, L^2)$       & ${\rm error}(p, L^2)$        \\
\midrule
\multirow{4}*{VEM}                             
&$1/10$       &$7.157458e-16$                       &$2.565404e-17$                             &$2.117754e-03$  \\
&$1/20$       &$1.524395e-15$                       &$2.597817e-17$                             &$5.489919e-04$  \\
&$1/40$       &$1.610876e-15$                       &$1.589614e-17$                             &$1.377769e-04$   \\
&$1/80$       &$9.630624e-15$                       &$4.590908e-17$                             &$3.465069e-05$  \\
\midrule
\multirow{4}*{Q2-P1}                        
&$1/10$       &$5.328708e-04$                       &$9.142870e-06$                             &$8.202921e-03$  \\
&$1/20$       &$1.486154e-04$                       &$1.278884e-06$                             &$2.623095e-03$  \\
&$1/40$       &$4.105136e-05$                       &$1.737273e-07$                             &$3.433991e-04$   \\
&$1/80$       &$1.006121e-05$                       &$2.164782e-08$                             &$8.511695e-05$  \\
\bottomrule
\end{tabular}
\caption{Test \ref{test1}:  Errors with VEM and Q2-P1 for polynomial pressure $p_1$ and meshes $\mathcal{Q}_h$.}
\label{tab1}
\end{table}
On the other hand we note that, due to the load approximation procedure, there is a load dependent term in the right hand side of \eqref{eq:errorvem}. As a consequence, in the test with goniometric pressure $p_2$ (where the load ${\bf f}$ is not a polynomial) we expect a slight pollution of the velocity errors also for the VEM scheme, although much smaller than for the FEM case.
In Figure \ref{fig:test1} we plot the errors for the goniometric pressure $p_2$ and the same sequence of meshes $\mathcal{Q}_h$. In accordance with the a-priori estimates \eqref{eq:errorfem}, \eqref{eq:errorvem} and the above observation, we obtain quadratic convergence rate for the Q2-P1 finite element method, and fourth order convergence rate for the VEM scheme for the $H^1$-velocity  (quadratic for the $L^2$-pressure errors).

\begin{figure}[!h]
\centering
\includegraphics[scale=0.25]{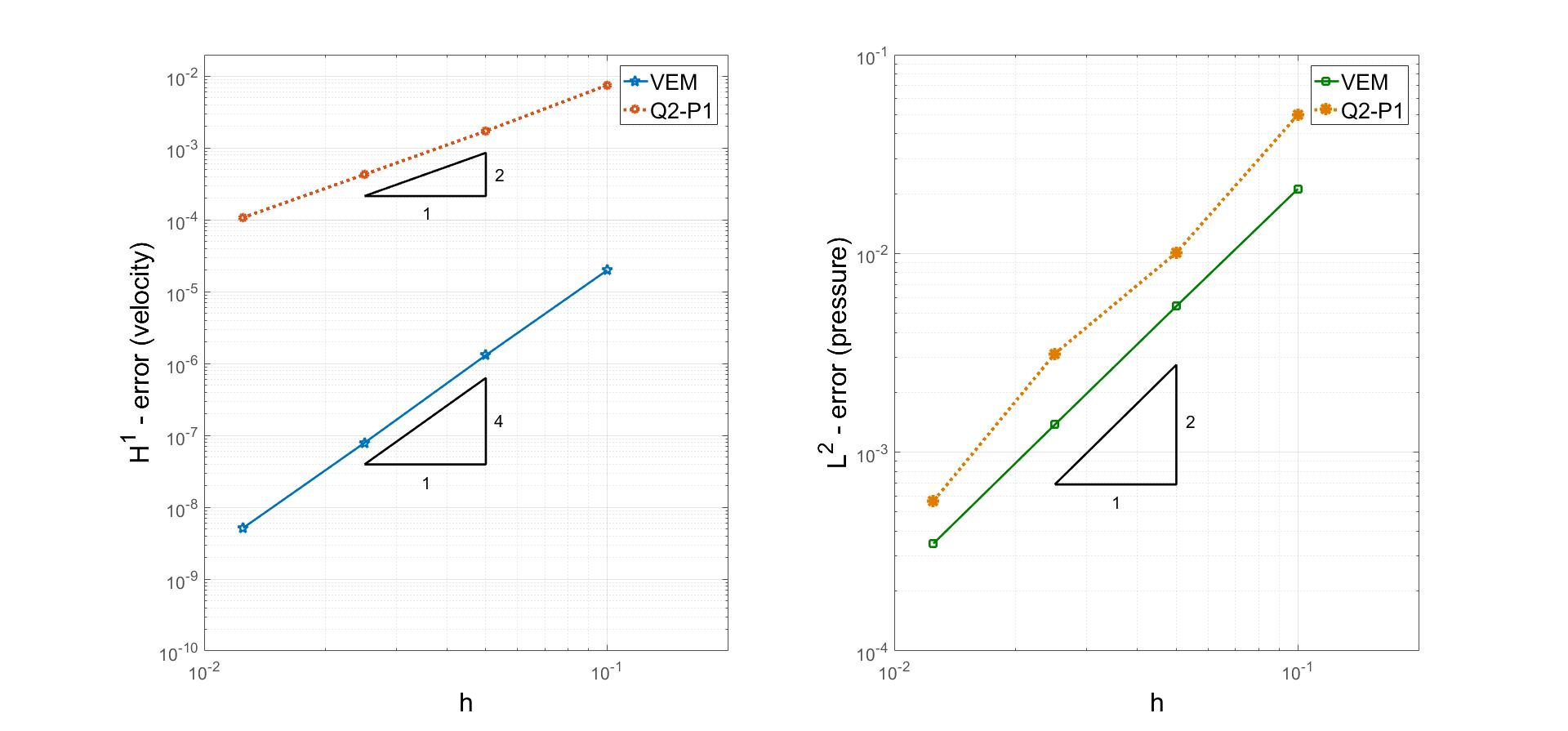} 
\caption{Test \ref{test1}:  Errors with VEM and Q2-P1 for the  meshes $\mathcal{Q}_h$.}
\label{fig:test1}
\end{figure}


\end{test}

\begin{test} [\textbf{Vanishing external load}]
\label{test2}
In this test we consider two benchmark Navier-Stokes problems taken from \cite{benchmark} on the disk $\Omega_{\rm D}$   where we compare the results obtained with VEM discretization with those obtained with the standard P2-P1 element for the sequence of meshes $\mathcal{T}_h$. The solutions are chosen in such a way that the pressures balance the nonlinear convective term yielding a vanishing external load $\ff = \mathbf{0}$.

In the first example we take $\nu = 1$ and the exact solution
\[
\uu_1(x, y) = \begin{pmatrix}-y \\  x \end{pmatrix} \qquad p_1(x, y) = -\frac{x^2 + y^2}{2} + \frac{1}{4}
\]  
We notice that the velocity $\uu_1$ belongs to the discrete space for both VEM and P2-P1 schemes.
In Table \ref{tab2} we show the results obtained with the P2-P1 element and the VEM discretization, in which we use respectively the trilinear form $c_h(\cdot; \, \cdot, \cdot)$ of \eqref{eq:c_h}, labelled as ${\rm VEM}_{\rm non-skew}$, and the skew-symmetric form $\widetilde{c}_h(\cdot; \, \cdot, \cdot)$ of \eqref{eq:ctilde_h}, labelled as ${\rm VEM}_{\rm skew}$ (cf. Remark \ref{rem:non-skew}).

\begin{table}[!h]
\centering
\begin{tabular}{ll*{3}{c}}
\toprule
&                         $h$                      & ${\rm error}(\mathbf{u}, H^1)$   & ${\rm error}(\mathbf{u}, L^{\infty})$       & ${\rm error}(p, L^2)$        \\
\midrule
\multirow{4}*{${\rm VEM}_{\rm non-skew}$}                             
&$ 1/5$       &$3.409332e-13$                       &$2.564615e-14$                             &$3.379691e-03$  \\
&$1/10$       &$8.055803e-13$                       &$3.158584e-14$                             &$8.512726e-04$  \\
&$1/20$       &$1.769002e-12$                       &$6.561418e-14$                             &$2.135981e-04$   \\
&$1/40$       &$4.080531e-12$                       &$8.147236e-14$                             &$5.352940e-05$  \\
\midrule
\multirow{4}*{${\rm VEM}_{\rm skew}$}                             
&$ 1/5$       &$5.738500e-05$                       &$3.252807e-06$                             &$3.379691e-03$  \\
&$1/10$       &$1.510897e-05$                       &$5.101225e-07$                             &$8.512726e-04$  \\
&$1/20$       &$3.438742e-06$                       &$7.243032e-08$                             &$2.135981e-04$   \\
&$1/40$       &$6.894319e-07$                       &$8.940261e-09$                             &$5.352940e-05$  \\    
\midrule
\multirow{4}*{P2-P1}                        
&$ 1/5$       &$4.658371e-04$                       &$4.031058e-05$                             &$3.416287e-03$  \\
&$1/10$       &$1.470468e-04$                       &$6.057622e-06$                             &$8.666102e-04$  \\
&$1/20$       &$2.760305e-05$                       &$7.740773e-07$                             &$2.160190e-04$   \\
&$1/40$       & fail to converge                                    &   fail to converge                          & fail to converge \\    
\bottomrule
\end{tabular}
\caption{Test \ref{test2}:  Errors with VEM and P2-P1 with solution $(\uu_1, p_1)$ and meshes $\mathcal{T}_h$.}
\label{tab2}
\end{table}
We observe that the ${\rm VEM}_{\rm non-skew}$ yields an exact solution $\uu_h = \uu_1$. Indeed, in this simple case it holds 
\[
c_h(\uu_1; \, \uu_1, \vv_h) = c(\uu_1; \, \uu_1, \vv_h) \qquad \text{for all $\vv_h \in \VV_h$.}
\]
This property is not verified by the skew-symmetric trilinear form $\widetilde{c}_h(\cdot; \, \cdot, \cdot)$. The P2-P1 does not yield the exact velocity solution since the velocity error of the method is polluted by the approximation of the pressure.

In the second example we set $\nu = 1$ and we consider the exact solution
\[
\uu_2(x, y) = 3\, \begin{pmatrix} x^2 - y^2 \\  -2 x y \end{pmatrix} \qquad p_2(x, y) = 9
\frac{(x^2 + y^2)^2}{2} - \frac{3}{2}
\]  
In Figure \ref{fig:test2tri} we show the results. 
\begin{figure}[!h]
\centering
\includegraphics[scale=0.25]{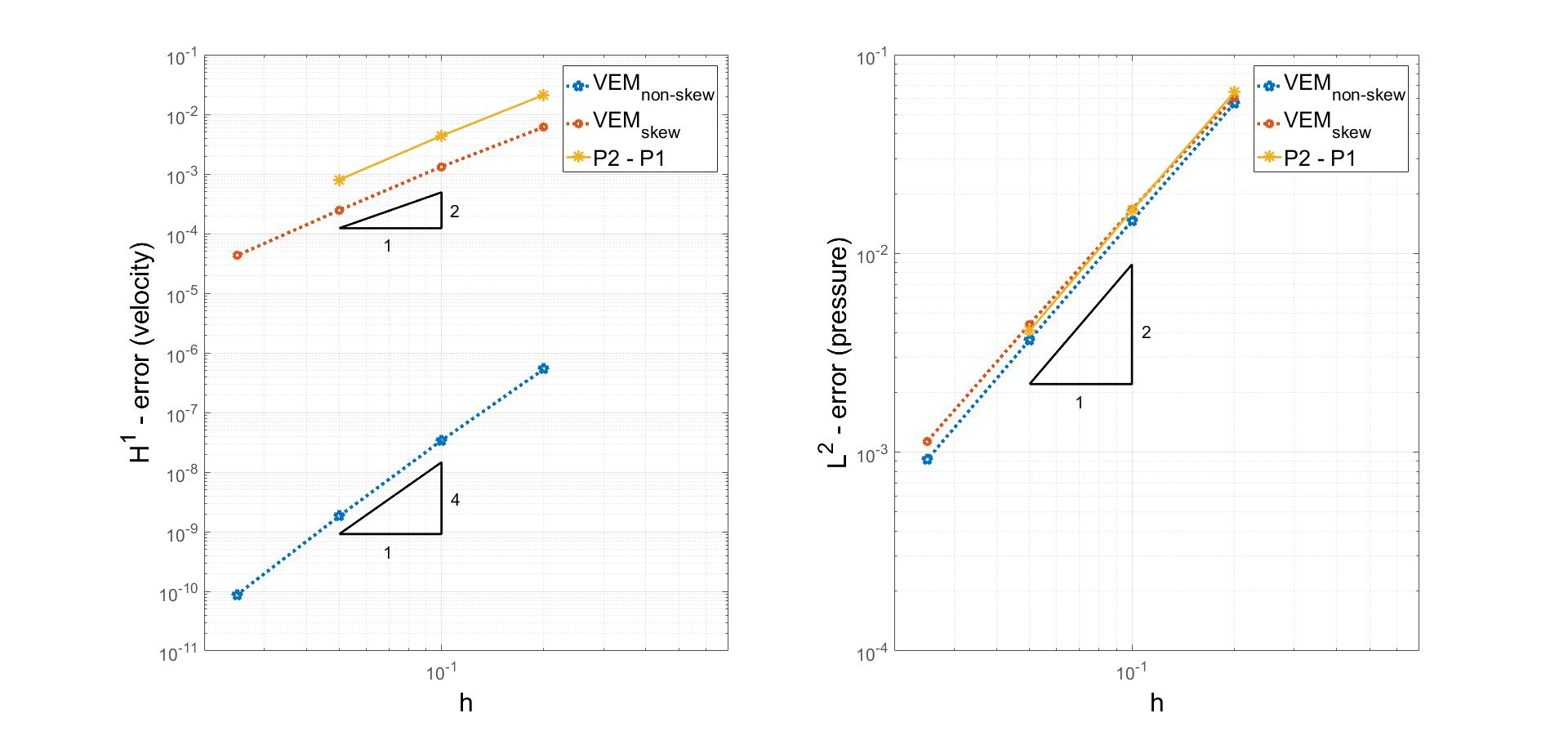} 
\caption{Test \ref{test2}:  Errors with VEM and P2-P1 with solution $(\uu_2, p_2)$ and meshes $\mathcal{T}_h$.}
\label{fig:test2tri}
\end{figure}

%
We are in a similar situation to the previous example (the velocity $\uu_2$ belongs to the discrete spaces, whereas the pressure $p_2$ does not) but with an important difference. 
Also in this case we observe that  ${\rm VEM}_{\rm non-skew}$ provides a better performance than ${\rm VEM}_{\rm skew}$, but now $\uu_h \neq \uu_2$. Indeed, in this case it holds
\[
c_h(\uu_2; \, \uu_2, \vv_h) - c(\uu_2; \, \uu_2, \vv_h) \leq C \, h^{k+2} |(\boldsymbol{\nabla}\uu )\, \uu|_{k+1} || \vv_h ||_{\VV} \qquad \text{for all $\vv_h \in \VV_h$}
\]
and using similar steps as in the proof of Theorem \ref{thm:u}, for  ${\rm VEM}_{\rm non-skew}$ we can derive
\[
\|\uu - \uu_h \|_{\VV} \leq C \, h^{k+2} \, \|(\boldsymbol{\nabla}\uu) \, \uu\|_{k+1} .
\]
Instead, for ${\rm VEM}_{\rm skew}$ we can only obtain
\[
\|\uu - \uu_h \|_{\VV} \leq C \, h^{k} \, |\uu \cdot \uu|_k.
\]
Finally, figure \ref{fig:test2vor} displays the results obtained with ${\rm VEM}_{\rm non-skew}$ and ${\rm VEM}_{\rm skew}$ for the sequence of polygonal meshes $\mathcal{V}_h$ (see Figure \ref{meshd}).
\begin{figure}[!h]
\centering
\includegraphics[scale=0.25]{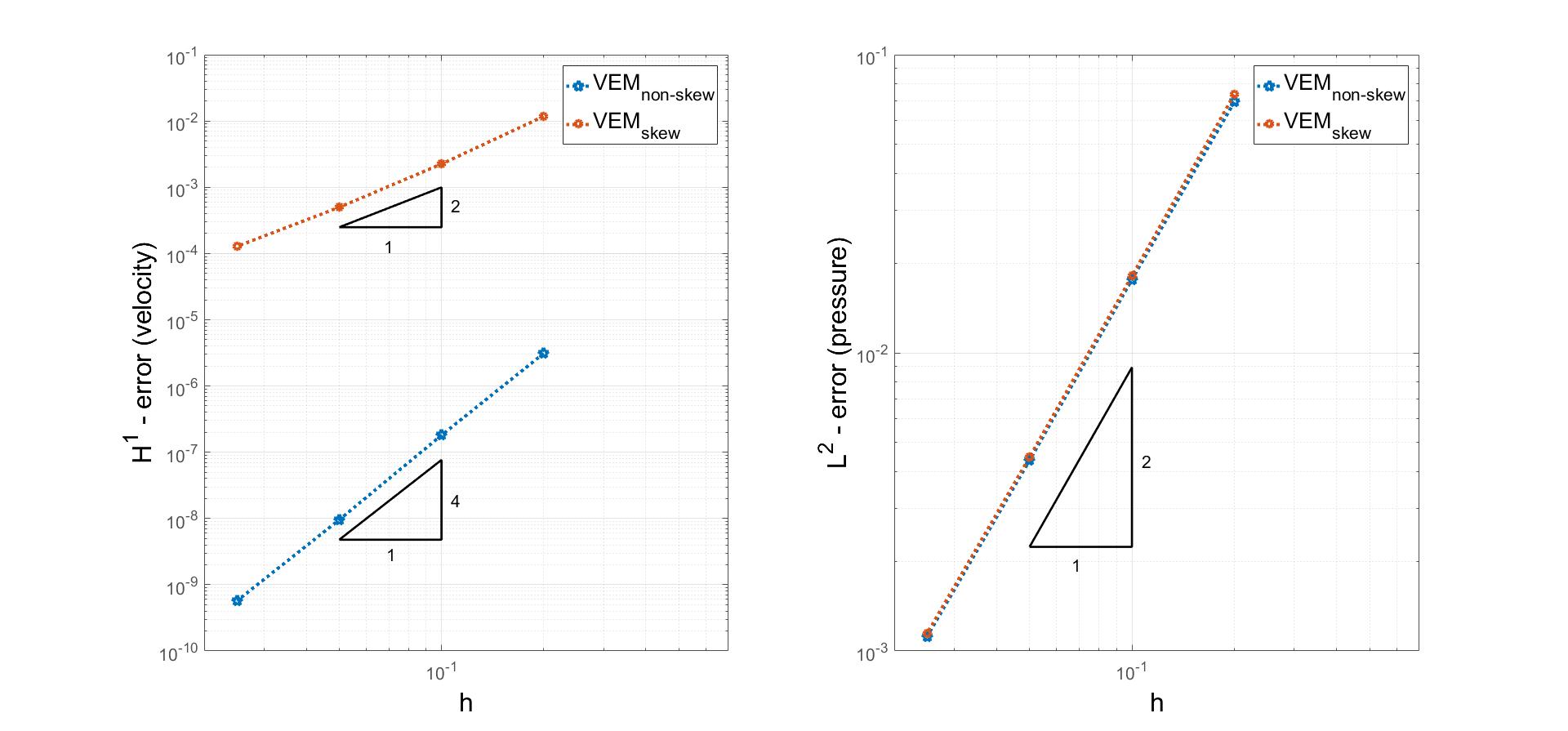} 
\caption{Test \ref{test2}:  Errors with VEM and P2-P1 with solution $(\uu_2, p_2)$ and meshes $\mathcal{V}_h$.}
\label{fig:test2vor}
\end{figure}

\end{test}

\begin{test}
\label{test3}
In this example we test the Navier-Stokes equation on the domain $\Omega_{\rm Q}$ with different values of the fluid viscosity $\nu$.  We choose the load term $\ff$ in such a way that the analytical solution is  
\[
\mathbf{u}(x,y) =  0.1 \, \begin{pmatrix}
x^2 (1 - x)^2 \, (2 y - 6 y^2 + 4 y^3)\\
- y^2 (1 - y)^2 \, (2 x - 6 x^2 + 4 x^3)
\end{pmatrix} 
\qquad 
p(x,y) = x^3 \, y^3 - \frac{1}{16}.
\]
The aim of this test is to check the actual performance of the virtual element method for small viscosity parameters, in comparison with the standard P2-P1 mixed finite element method.  
Figure \ref{fig:test3} shows that the solutions of the virtual element method are accurate even for rather
small values of $\nu$. Larger velocity errors appear only for very small viscosity parameters. The reason for this robustness is again that the ``divergence free'' property of VEM yields velocity errors that do not depend directly on the pressure (but only indirectly through the higher order load approximation term, see Theorem \ref{thm:u}). On the contrary, for the P2-P1 element the pressure component of the error can become the dominant source of error also for the velocity field. In addition, we note that for $\nu=10^{-4}, 10^{-5}$ the P2-P1 element does not even converge. 

\begin{figure}[!h]
\centering
\includegraphics[scale=0.25]{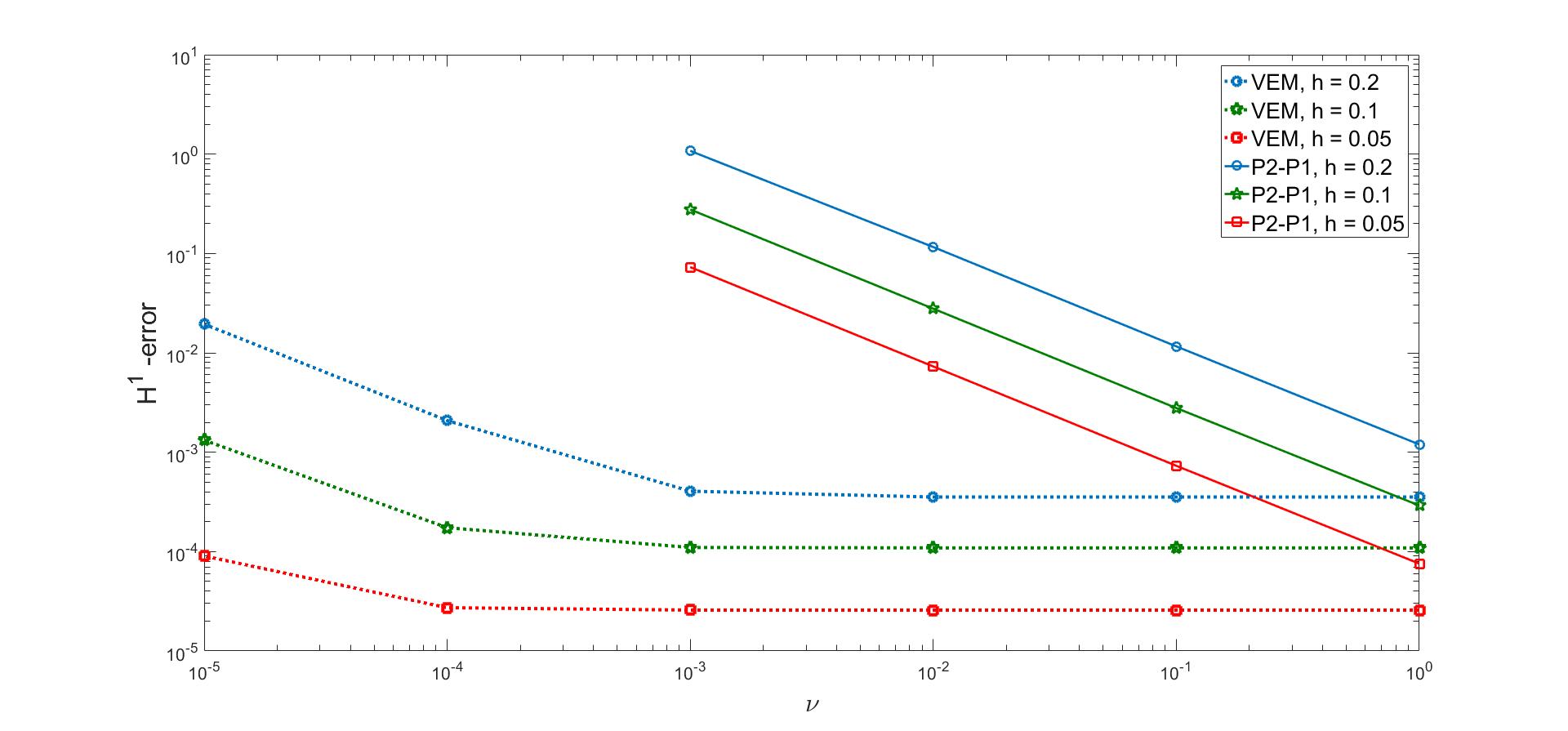} 
\caption{Test \ref{test3}:  Errors of VEM (dotted lines) and P2-P1 (solid lines), with different values of $\nu$ for the  meshes $\mathcal{T}_h$.}
\label{fig:test3}
\end{figure}
\end{test}

\begin{test}
\label{test4}
In this test we solve the Navier-Stokes equation on the square domain $\Omega_{\rm Q}$ with  viscosity $\nu = 0.1$ and with
the load term $\mathbf{f}$ chosen such that the analytical solution is 
\[
\mathbf{u}(x,y) = \frac{1}{2} \, \begin{pmatrix}
\sin(2 \pi x)^2 \, \sin(2 \pi y) \, \cos(2 \pi y) \\
- \sin(2 \pi y)^2 \, \sin(2 \pi x) \, \cos(2 \pi x)
\end{pmatrix} \qquad 
p(x,y) = \pi^2 \, \sin(2 \pi x) \, \cos(2 \pi y).
\]
In Figure \ref{fig:test4} we show the results obtained for the sequence of triangular meshes $\mathcal{T}_h$, also compared with the P2-P1 element.
\begin{figure}[!h]
\centering
\includegraphics[scale=0.25]{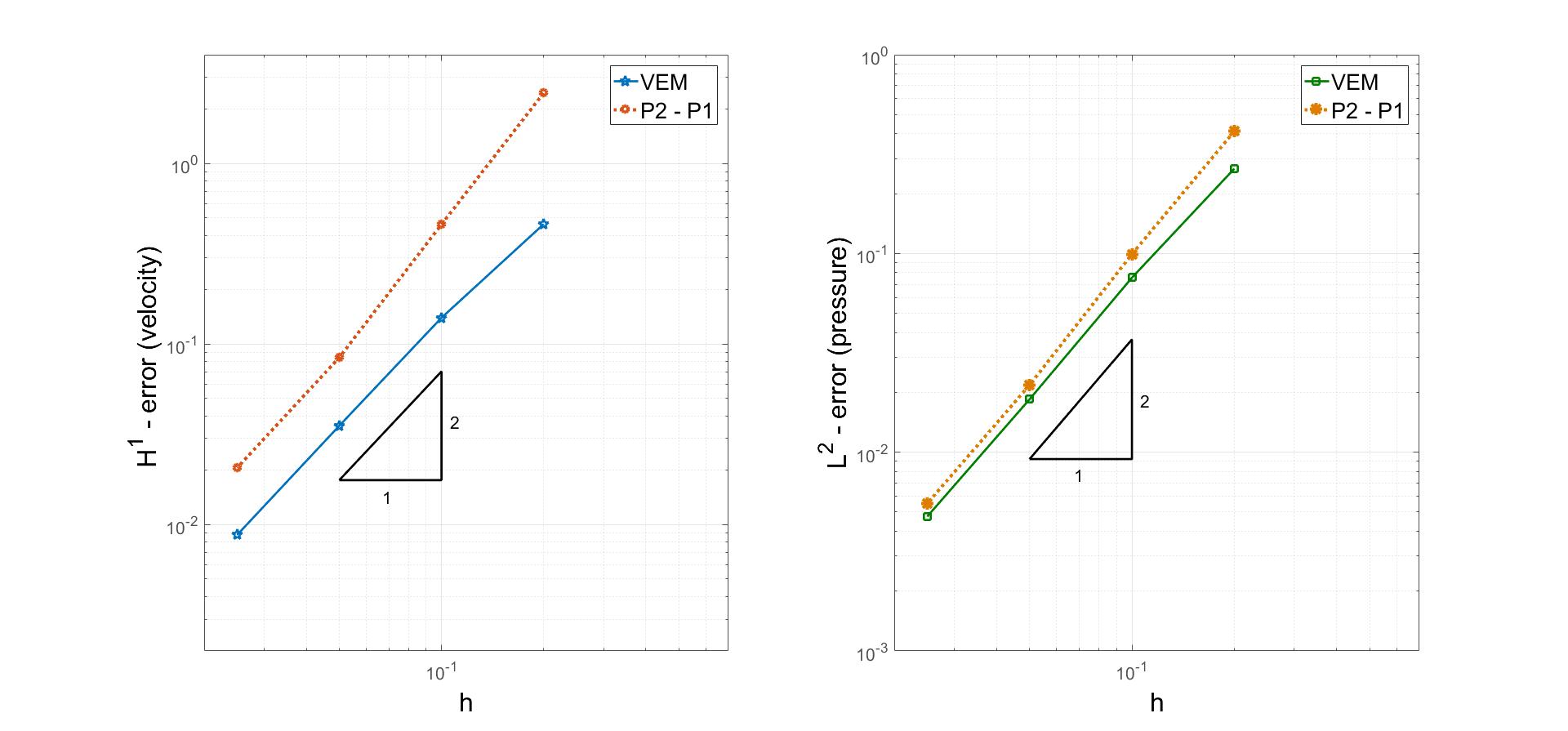}
\caption{Test \ref{test4}:  Errors with VEM and P2-P1 for the  meshes $\mathcal{T}_h$.}
\label{fig:test4}
\end{figure}

%
We notice that the theoretical predictions of Sections \ref{sec:4} are confirmed. Moreover, we observe that the virtual element method exhibit smaller errors than the standard P2-P1 method, at least for this example and with the adopted meshes.
Finally we test the virtual element method with the sequence of polygonal meshes $\mathcal{W}_h$, obtaining that the theoretical results are confirmed as well (note that the $N_{dof}$ behaves like $h^{-2}$).

\begin{figure}[!h]
\centering
\includegraphics[scale=0.25]{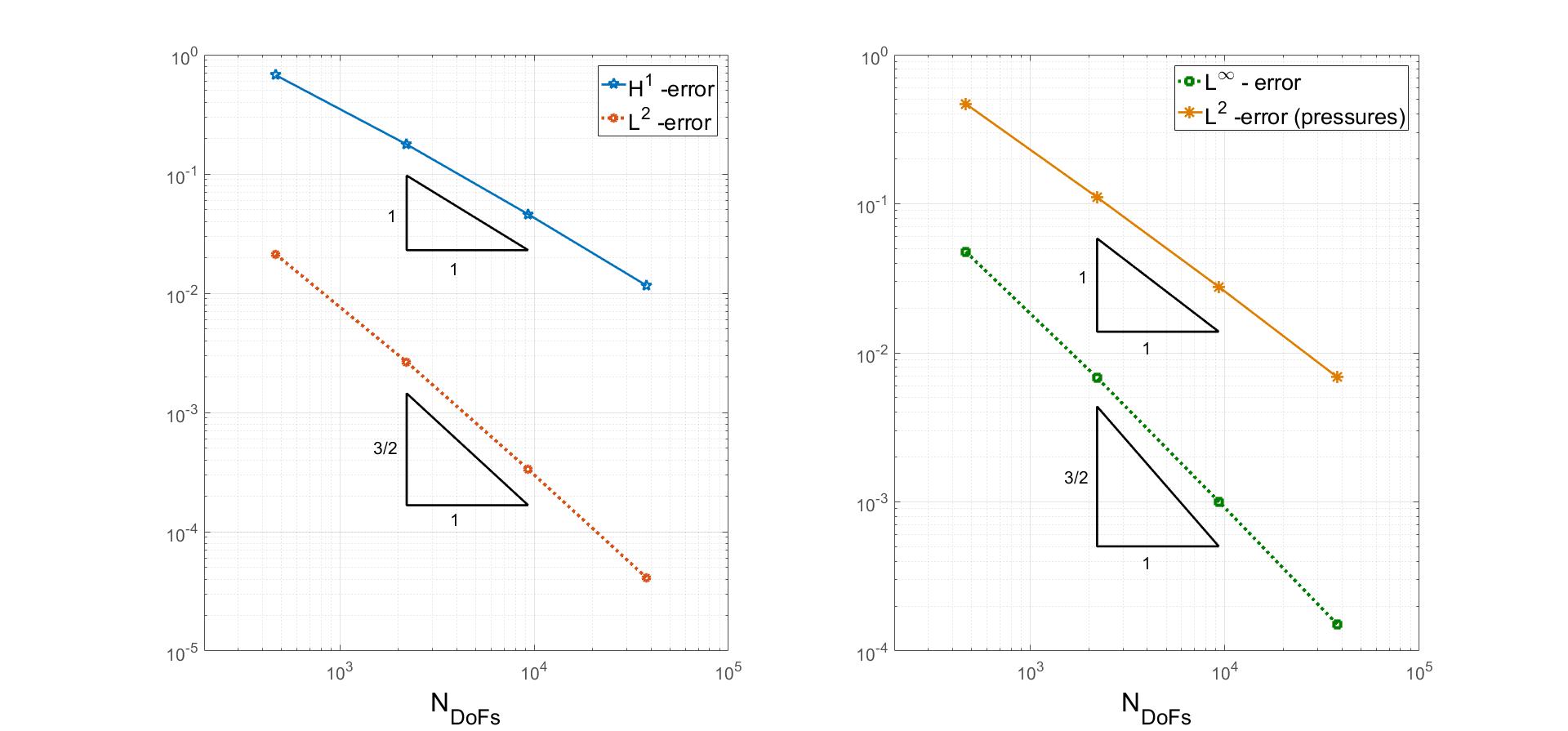} 
\caption{Test \ref{test4}:  Errors with VEM  for the  meshes $\mathcal{W}_h$.}
\label{fig:test4_esa}
\end{figure}
\end{test}

\begin{test}
\label{test5}
In this experiment we analyse the Stokes equation on the square domain $\Omega_{\rm Q}$ where the viscosity $\nu = 1$ and
the load term $\mathbf{f}$ is chosen such that the analytical solution is 
\[
\mathbf{u}(x,y) = \frac{1}{2} \, \begin{pmatrix}
\sin(2 \pi x)^2 \, \sin(2 \pi y) \, \cos(2 \pi y) \\
- \sin(2 \pi y)^2 \, \sin(2 \pi x) \, \cos(2 \pi x)
\end{pmatrix} \qquad 
p(x,y) =  \sin(2 \pi x) \, \cos(2 \pi y).
\] 
The aim of this test is the assessment of the VEM robustness with respect to the mesh deformation, performing also a comparison with the Q2-P1 mixed finite element method.
In Figures \ref{fig:test5_storti} and \ref{fig:test5_ultra_storti} we plot the obtained  results.
\begin{figure}[!h]
\centering
\includegraphics[scale=0.25]{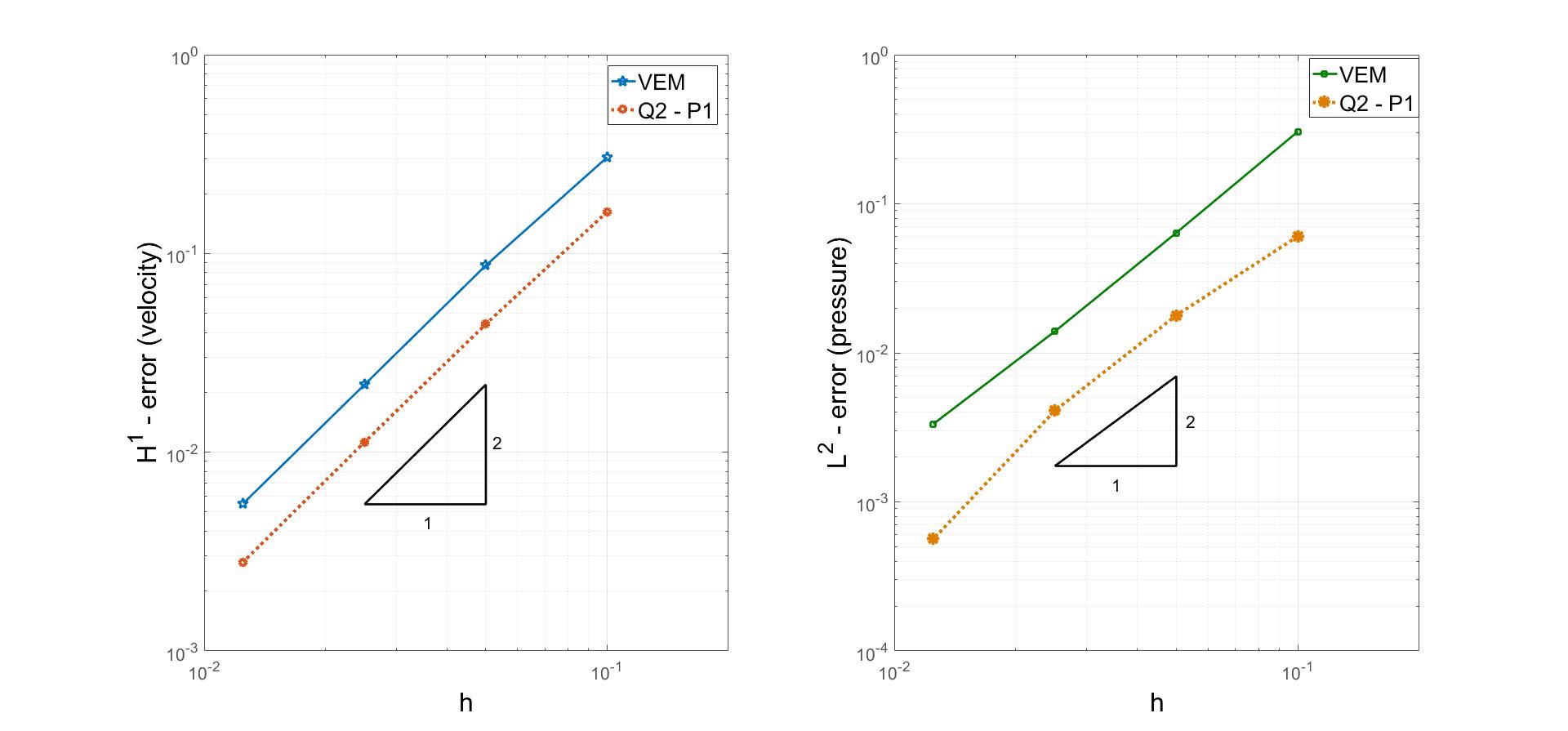} 
\caption{Test \ref{test5}:  Errors with VEM and Q2-P1 for the  meshes $\mathcal{Q}_h$.}
\label{fig:test5_storti}
\end{figure}

\begin{figure}[!h]
\centering
\includegraphics[scale=0.25]{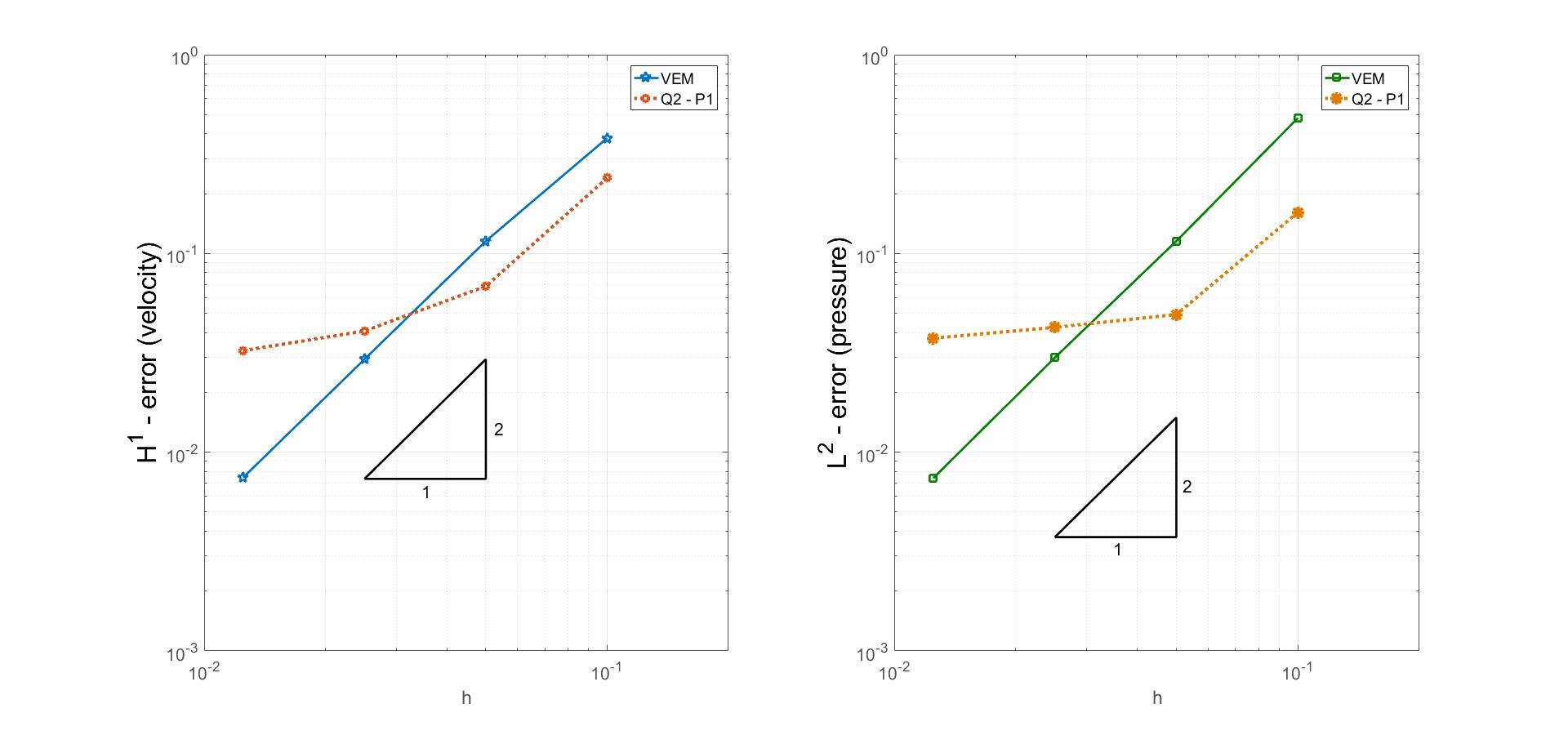} 
\caption{Test \ref{test5}:  Errors with VEM and Q2-P1 for the  meshes $\mathcal{U}_h$.}
\label{fig:test5_ultra_storti}
\end{figure}



\end{test}

We observe that for the (less deformed) quadrilateral meshes $\mathcal{Q}_h$, both the virtual element method and the Q2-P1 preserve the theoretical order of accuracy, but the Q2-P1 element yields better results. 
Instead, for the  (more deformed) sequence of meshes $\mathcal{U}_h$, the behaviour is completely different. The virtual element approach maintains the optimal second order accuracy, whereas the Q2-P1 element clearly suffers from an evident sub-optimality of the convergence rates (the pressure does not even seem to converge). Therefore, we may conclude that the VEM seems to be more robust with respect to large distortions of the mesh.


\begin{test}
\label{test6}
This test highlights that, following \cite{Stokes:divfree,preprintdarcy}, the proposed virtual elements can accommodate both the Stokes (or Navier-Stokes) and the Darcy problems simultaneously (see Remark \ref{rem:Sto-Dar}). 
Accordingly, we consider the approximation of a flow in the square $[0, 2]^2$,
consisting of a porous region $\Omega_D := \Omega_{D1} \cup \Omega_{D2}$, where the flow is a Darcy flow, and an open region $\Omega_S = \Omega \setminus \Omega_D$,
where the flow is governed by the linear Stokes system (see Figure \ref{fig:test6_problem} for a depiction of the problem configuration).
This leads to consider the problem: find $(\uu,p) \in [H^1(\Omega)]^2 \times L^2(\Omega)$ such that 

\begin{equation}
\label{eq:coupled}
\left\{
\begin{aligned}
& -  2 \nu \, \dd (\epsilon (\uu))  -  \nabla p = \mathbf{0} \quad  & &\text{in $\Omega_S$,} \\
& \dd \, \uu = 0 \quad & &\text{in $\Omega_S$,} \\
& \uu_1 = \phi,  \quad & &\text{on $\{0\} \times [0, 2]$,} \\
& \uu_2 = 0  \quad & &\text{on $[0, 1] \times \{0, 2\}$,} \\
\end{aligned}
\right.
\qquad 
\left\{
\begin{aligned}
&  \nu \, \lambda \uu -  \nabla p =  \mathbf{0} \quad  & &\text{in $\Omega_D$,} \\
& \dd \, \uu = 0 \quad & &\text{in $\Omega_D$,} \\
& \uu_2 = 0  \quad & &\text{on $[1, 2] \times \{0, 2\}$,} \\
\end{aligned}
\right.
\end{equation}
where $\epsilon(\uu) := \frac{1}{2}({\Gr} \uu + {\Gr} \uu^T)$ denotes the symmetric gradient operator. We fix $\nu =1$, $\lambda= 10$ on $\Omega_{D1}$, and $\lambda= 2$ on $\Omega_{D2}$. Furthermore, we take
\[
\phi(0,y) = \max \{0, -10(1-y)(2-y)\}.
\]
At the interface between Stokes and Darcy regions, the system \eqref{eq:coupled} is coupled using the Beavers-Joseph-Saffmann condition (see \cite{beavers1967boundary,saffman1971boundary} for further details).

We observe that in our test problem, we set free boundary conditions on the right boundary edge of the Darcy region. 
To test the performance of the virtual element method, we compute the unknown fluxes quantity (see Figure \ref{fig:test6_problem})
\[
f_{R1} := \int_{\partial \Omega \cap \partial \Omega_{D1}} \uu \cdot \mathbf{n} \, {\rm d}s \qquad \text{and} \qquad
f_{R2} := \int_{\partial \Omega \cap \partial \Omega_{D2}} \uu \cdot \mathbf{n} \, {\rm d}s
\]
taking into account that $f_L + f_{R1} + f_{R2} = 0$ and 
\[
f_{L} := \int_{\partial \Omega \cap \partial \Omega_{S}} \uu \cdot \mathbf{n} \, {\rm d}s = \frac{10}{6} ,
\]
with $\mathbf{n}$ denoting the outward normal vector.
 
\begin{figure}[!h]
\centering
\includegraphics[scale=0.6]{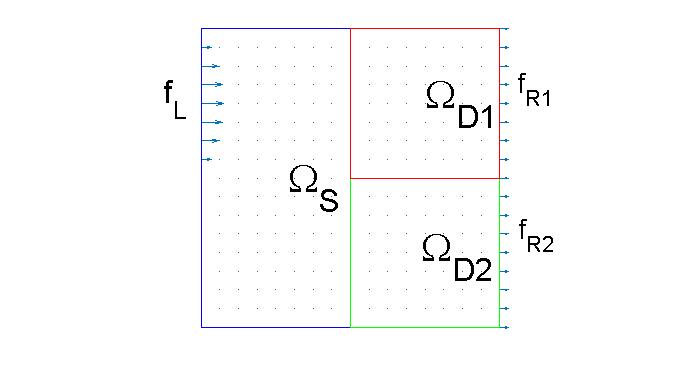} 
\caption{Test \ref{test6}:  The domain configuration of Problem \eqref{eq:coupled}.}
\label{fig:test6_problem}
\end{figure}

In this experiment the computational domain $\Omega:= [0,2]^2$  is partitioned using two sequences of polygonal meshes:
\begin{itemize}
\item $\{ \mathcal{Q}_{h} \}_h$: sequence of square meshes with element edge length $h=1/4, 1/8, 1/16, 1/32$ ;
\item $\{ \mathcal{P}_h\}_h$: sequence of meshes obtained by gluing a Voronoi decomposition on the domain $\Omega_S$, a triangular decomposition on $\Omega_{D1}$ and a square decomposition on $\Omega_{D2}$, with edge length $h=1/4, 1/8, 1/16, 1/32$.
\end{itemize}
In addition, we use the square mesh $\mathcal{Q}_{1/64}$ as the basis for the reference solution. 
An example of the adopted meshes is shown in Figure \ref{fig:test6_mesh}. In Figure \ref{fig:test6_velocity} 
we show the plot of the numerical velocity and pressure. Note that the purpose of mesh family $\{ \mathcal{P}_h\}_h$ is to show the robustness of the proposed method when, by exploiting the flexibiliy of polygonal grids, completely independent meshes are glued together.
\begin{figure}[!h]
\centering
\includegraphics[scale=0.8]{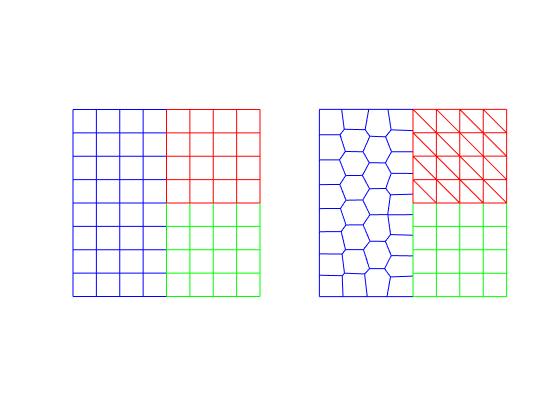}
\vspace{-2.0cm}
\caption{Test \ref{test6}:  Example of the adopted polygonal meshes: $\mathcal{Q}_{1/4}$ (left), $\mathcal{P}_{1/4}$ (right).}
\label{fig:test6_mesh}
\end{figure}
\begin{figure}[!h]
\centering
\includegraphics[scale=0.28]{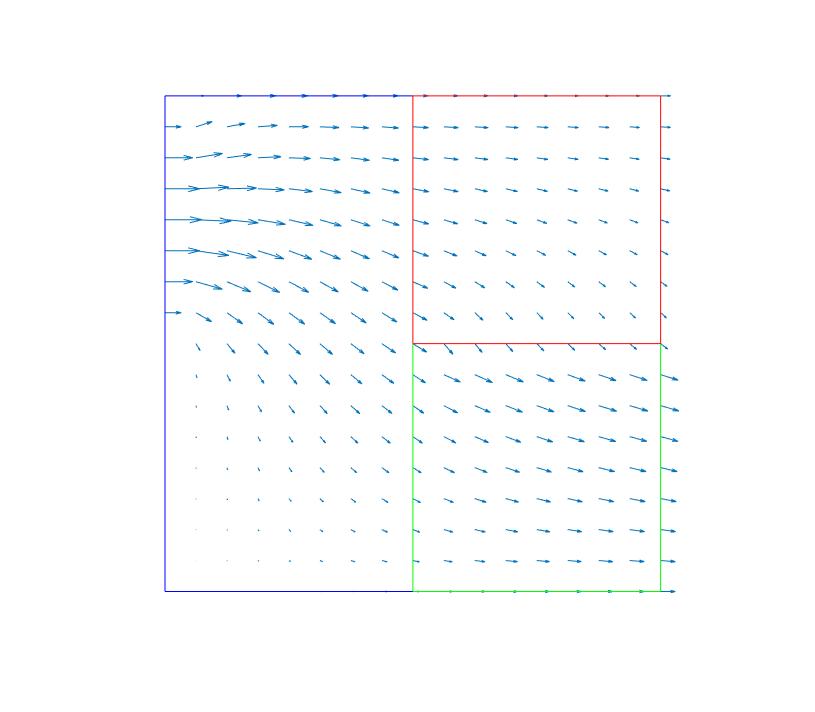} 
\includegraphics[scale=0.23]{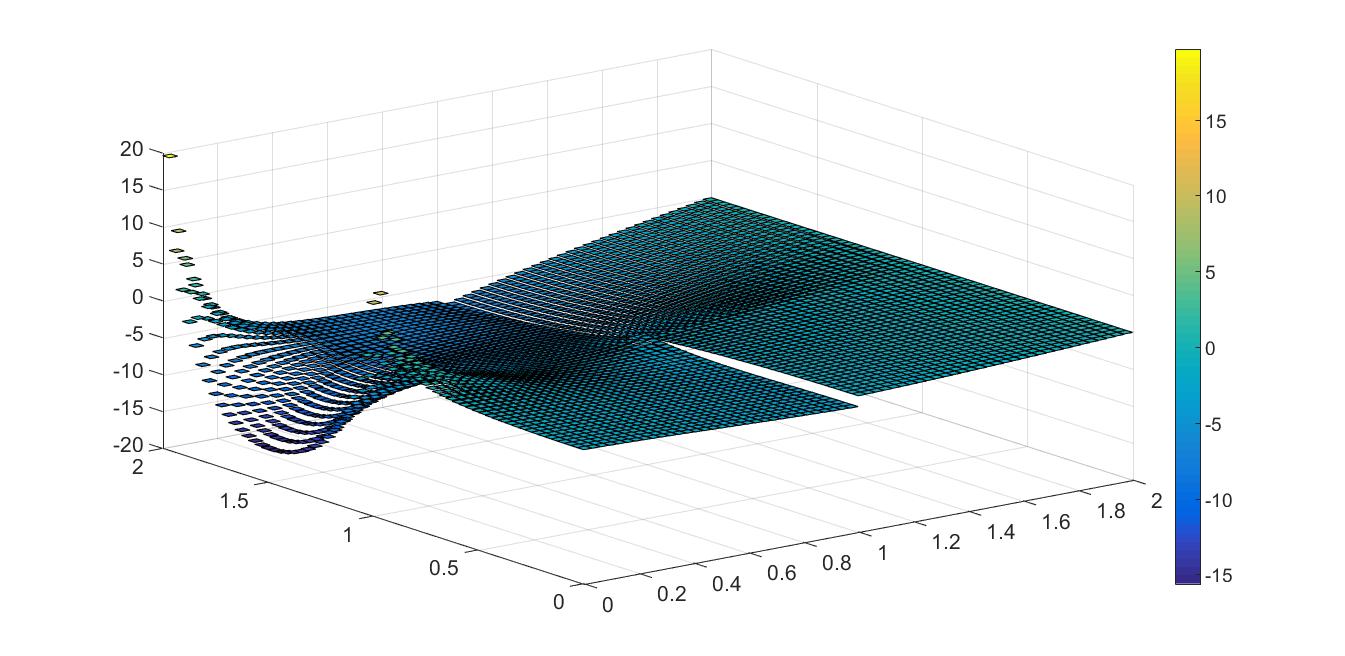} 
\vspace{-0.5cm}
\caption{Test \ref{test6}:  Velocity and pressure respectively for the mesh $\mathcal{Q}_{1/8}$ and $\mathcal{Q}_{1/32}$.}
\label{fig:test6_velocity}
\end{figure}
In Table \ref{tab6} we show the results obtained by using the sequences of meshes $\mathcal{Q}_h$ and $\mathcal{P}_h$, compared with those obtained with the reference mesh.
We observe that both sequences of meshes exhibit appropriate convergence properties, confirming that the proposed virtual element method 
can automatically handle non-conforming polygonal meshes and the coupling between Darcy and Stokes flow problems.

\begin{table}[!h]
\centering
\begin{tabular}{ll*{3}{c}}
\toprule
&                         $h$                      & $f_{R1}$   & $f_{R2}$       & $f_{L} + f_{R1} + f_{R2}$        \\
\midrule
\multirow{1}*{reference mesh}
&$1/64$      &$5.373938e-01$                       &$1.129272e+00$                               &-$1.332267e-15$  \\
\midrule
\multirow{4}*{$\mathcal{Q}_h$}                             
&$1/4$       &$5.215469e-01$                       &$1.145119e+00$                             &$0$  \\
&$1/8$       &$5.284186e-01$                       &$1.138248e+00$                             &$4.440892e-16$  \\
&$1/16$      &$5.339269e-01$                       &$1.132739e+00$                             &-$4.44089e-16$   \\
&$1/32$      &$5.367736e-01$                       &$1.129893e+00$                             &$6.661338e-16$  \\   
\midrule
\multirow{4}*{$\mathcal{P}_h$}                        
&$1/4$       &$5.161923e-01$                       &$1.158934e+00$                             &$4.440892e-16$  \\
&$1/8$       &$5.254995e-01$                       &$1.143962e+00$                             &-$4.44089e-16$  \\
&$1/16$       &$5.343364e-01$                      &$1.132836e+00$                             &-$2.22044e-16$   \\
&$1/32$       &$5.381031e-01$                      &$1.128622e+00$                             &$0$  \\               
\bottomrule
\end{tabular}
\caption{Test \ref{test6}:  Fluxes along the boundary for the sequences of meshes $\mathcal{Q}_h$ and $\mathcal{P}_h$.}
\label{tab6}
\end{table}

\end{test}

\section{Acknowledgements}
The authors L. Beir\~ao da Veiga and G. Vacca were partially supported by the European Research Council through
the H2020 Consolidator Grant (grant no. 681162) CAVE, Challenges and Advancements in Virtual Elements. This support is gratefully acknowledged.

\addcontentsline{toc}{section}{\refname}
\bibliographystyle{plain}
\bibliography{biblio}

\end{document}